\documentclass[12pt]{amsart}

\usepackage{hyperref}
\usepackage{amsthm,amssymb,fullpage,tabu,tikz,ytableau}
\theoremstyle{definition}
\newtheorem{thm}{Theorem}[section]

\newtheorem{prop}[thm]{Proposition}

\newtheorem{conj}[thm]{Conjecture}

\newtheorem{lemma}[thm]{Lemma}
\newtheorem{defn}[thm]{Definition}
\newtheorem{cor}[thm]{Corollary}

\newtheorem{question}[thm]{Question}

\newtheorem{remark}[thm]{Remark}

\DeclareMathOperator{\aff}{aff}

\DeclareMathOperator{\AltAv}{\widetilde{Av}}
\DeclareMathOperator{\Av}{Av}

\DeclareMathOperator{\cn}{cone}
\DeclareMathOperator{\conv}{conv}
\DeclareMathOperator{\CRY}{CRY}
\DeclareMathOperator{\Des}{Des}

\DeclareMathOperator{\des}{des}
\DeclareMathOperator{\face}{face}
\DeclareMathOperator{\gglex}{\succ_{grevlex}}
\DeclareMathOperator{\lglex}{\prec_{grevlex}}

\DeclareMathOperator{\Irr}{Irr}
\DeclareMathOperator{\inv}{inv}
\DeclareMathOperator{\Invv}{Invv}
\DeclareMathOperator{\ehr}{\mathcal{L}}
\DeclareMathOperator{\grob}{\mathcal{G}}
\DeclareMathOperator{\init}{\mathrm{in}_{\prec}}
\DeclareMathOperator{\Inv}{\mathrm{Inv}}
\DeclareMathOperator{\pos}{pos}
\DeclareMathOperator{\rad}{rad}

\DeclareMathOperator{\vol}{vol}
\DeclareMathOperator{\Vol}{Vol}

\newcommand{\A}{\mathcal{A}}
\newcommand{\AltB}{\widetilde{B}}

\newcommand{\AltQ}{\widetilde{Q}}
\newcommand{\AltT}{\widetilde{\mathcal{T}}}

\newcommand{\FF}{\mathcal{F}}
\newcommand{\HH}{\mathcal{H}}

\newcommand{\RR}{\mathbb{R}}
\newcommand{\cR}{\mathcal{R}}

\newcommand{\ZZ}{\mathbb{Z}}
\newcommand{\symm}{\mathfrak{S}}
\newcommand{\TT}{\mathcal{T}}
\newcommand{\LL}{\mathcal{L}}

\newcommand{\dprime}{\prime\prime}
\newcommand{\da}{\hspace{-2pt}\downarrow}
\newcommand{\cF}{\mathcal F}
\newcommand{\iso}{\cong}

\newcommand{\OO}{\mathcal{O}}
\newcommand{\om}{\omega}
\newcommand{\Om}{\Omega}
\newcommand{\ol}{\overline}
\newcommand{\sig}{\sigma}
\newcommand{\la}{\lambda}
\newcommand{\sigp}{\sigma^{\prime}}

\newcommand{\llex}{<_{\rm lex}}
\newcommand{\glex}{>_{\rm lex}}

\usetikzlibrary{decorations.markings}
\tikzset{->-/.style={decoration={
  markings,
  mark=at position .5 with {\arrow[scale=1.5]{>}}},postaction={decorate}}}

\begin{document}

\title{Pattern-Avoiding Polytopes}
\author{Robert Davis and Bruce Sagan}

\begin{abstract}
	Two well-known polytopes whose vertices are indexed by permutations in the symmetric group $\symm_n$ are 
	the permutohedron $P_n$ and the Birkhoff polytope $B_n$.  We consider polytopes $P_n(\Pi)$ and $B_n(\Pi)$, whose vertices correspond to the permutations in $\symm_n$ avoiding a set of patterns $\Pi$.
	For various choices of $\Pi$, we explore the Ehrhart polynomials and $h^*$-vectors of these polytopes as well as other aspects of their combinatorial structure.
		
	For $P_n(\Pi)$, we consider all subsets $\Pi \subseteq \symm_3$ and are able to provide results in most cases.
	 To illustrate, $P_n(123,132)$ is  a Pitman-Stanley polytope, the number of  interior lattice points in $P_n(132,312)$  is a derangement number,
	and the normalized volume of $P_n(123,231,312)$ is the number of trees on $n$ vertices.

	The polytopes $B_n(\Pi)$ seem much more difficult to analyze, so we focus on four particular choices of $\Pi$.
	First we show that the $B_n(231,321)$ is exactly the Chan-Robbins-Yuen polytope.  Next we prove that for any $\Pi$ containing $\{123,312\}$ we have $h^*(B_n(\Pi))=1$.
	Finally, we  study $B_n(132,312)$ and $\AltB_n(123)$, where the tilde indicates that  we choose vertices corresponding to  alternating permutations avoiding the pattern $123$.
	In both cases we use order complexes of posets and techniques from toric algebra to construct regular, unimodular triangulations of the polytopes.
	The posets involved turn out to be isomorphic to the lattices of Young diagrams contained in a certain shape, and this permits us to give an exact expression for the normalized volumes of the corresponding polytopes via the hook formula.
	Finally, Stanley's theory of  $(P,\omega)$-partitions allows us to show that their $h^*$-vectors are symmetric and unimodal.
	
	Various questions and conjectures are presented throughout.
\end{abstract}
	
\maketitle

%%%%%%%%%%%%%%%%%%%%%%%%
%%%%%%%%%%%%%%%%%%%%%%%%

\section{Introduction}

Let $\symm_n$ denote the symmetric group on $1,2,\dots,n$ and $\symm = \cup_{n\ge0}\symm_n$.
Let $\pi  \in \symm_k$ and $\sig \in \symm_n$.
We say that $\sig$ {\em contains the pattern} $\pi$ if there is some substring $\sig'$ of $\sig$ whose elements have the same relative order as those in $\pi$.
Alternatively, we view $\sig'$ as {\em standardizing} to $\pi$ by replacing the smallest element of $\sig'$ with $1$, the next smallest by $2$, and so on.
If there is no such substring then we say that $\sig$ {\em avoids the pattern} $\pi$.
If $\Pi \subseteq \symm$, then we say $\sig$ {\em avoids} $\Pi$ if $\sig$ avoids every element of $\Pi$.
We will use the notation 
\[\Av_n(\Pi) := \{\sig \in \symm_n\ |\ \sig \text{ avoids } \Pi \}.\]
Note this is {\em not} the avoidance class of $\Pi$ which is the union of these sets over all $n$.

A polytope $P \subseteq \RR^n$ is the convex hull of finitely many points, written $P = \conv\{v_1,\ldots, v_k\}$.
Equivalently, a polytope may be described as a bounded intersection of finitely many half-spaces.
The {\em dimension} of $P$ is the dimension of its affine span.
We think of vectors in $\RR^n$ as columns and use $a^Tb$ to denote the usual inner product of $a,b\in\RR^n$.
An affine hyperplane $H$ determined by the equation $a^Tx = b$ for some $a,b \in \RR^n$ is called {\em supporting} if $a^Tp \geq b$ for every $p \in P$.
Some texts, such as \cite{GrunbaumConvexPolytopes}, insist that $H \cap P$ be nonempty; our definition aligns with those found in \cite{BeckRobinsCCDed2, StanleyVol1Ed2}.
If $H$ is a supporting hyperplane, then the set $H \cap P$ is called a {\em face} of $P$ and is a subpolytope of $P$.
Faces of dimension $0$ are {\em vertices}, faces of dimension $1$ are called {\em edges}, and faces of dimension $\dim P - 1$ are called {\em facets}.
Additionally, we say a polytope is a {\em lattice polytope}  if each vertex is an element of $\ZZ^n$.
Lattice polytopes have long found connections with permutations, in particular via the permutohedron and Birkhoff polytope.

The {\em permutohedron} is defined as 
\[P_n := \conv\{(a_1,\ldots, a_n)\ |\  a_1\cdots a_n \in \symm_n\}.\]
We will often make no distinction between a permutation and its corresponding point in $\RR^n$.
This polytope was first described in~\cite{Schoute} and has connections to the geometry of flag varieties as well as representations of $GL_n$.
We refer to~\cite{ZieglerLectures} for general background regarding permutohedra.

The {\em Birkhoff polytope} is the polytope 
\[
	B_n := \conv\left\{X = (x_{i,j}) \in (\RR_{\geq 0})^{n\times n} \ |\ \sum_{i=1}^n x_{i,j} = \sum_{j=1}^n x_{i,j} = 1 \text{ for all } i,j\right\}.
\]
The Birkhoff-von Neumann Theorem states that the vertices of $B_n$ are the permutation matrices.

In this article, we describe a natural blending of pattern avoidance with  the permutohedron and the Birkhoff polytope. 
Specifically, for any set of patterns $\Pi$, we define $P_n(\Pi)$ to be the subpolytope of $P_n$ obtained by taking the convex hull of those vertices corresponding to permutations in $\Av_n(\Pi)$.  
The polytope $B_n(\Pi)$ is defined similarly.  
We study the Ehrhart polynomials and $h^*$-vectors of these polytopes as well as other aspects of their combinatorial structure.

The rest of this paper is organized as follows.
In Section~\ref{sec:pre} we review some basic notions about  pattern avoidance and polytopes which will be needed throughout.
Section~\ref{sec:permutohedra} focuses on the permutohedron case $P_n(\Pi)$.
We first show in Proposition~\ref{prop: Pn equivalence} that the action of a certain subgroup of the dihedral group of the square produces unimodularly equivalent polytopes.  
We then consider all possible $\Pi\subseteq\symm_3$ and are able to provide results for most of the orbits of this action.  
Specific propositions are listed in Table~\ref{tab: summary}.  
As a sampling, $P_n(123,132)$ is  a Pitman-Stanley polytope, the number of interior lattice points in $P_n(132,312)$  is a derangement number, and the normalized volume of $P_n(123,231,312)$ is the number of trees on $n$ vertices.

The $\Pi$-avoiding Birkhoff polytope appears to be much harder to analyze in general.  So we concentrate  on four specific examples.  
In Section~\ref{sec:birkhoff}, we show that $B_n(231,321)$ is a polytope studied by Chan, Robbins, and Yuen.
Next we prove that for any $\Pi$ containing  the permutations $123$ and $312$ we have $h^*(B_n(\Pi))=1$.
	In Section~\ref{sec:132,312 and 123} we begin our study of $B_n(132,312)$ and $\AltB_n(123)$, the tilde indicating that  we choose vertices corresponding to  alternating permutations avoiding the pattern $123$.
	In both cases we use order complexes of posets and techniques from toric algebra to construct regular, unimodular triangulations of the polytopes.
	The posets involved turn out to be isomorphic to the lattices of Young diagrams contained in a certain shape, and this permits us to give an exact expression for the normalized volumes of the corresponding polytopes via the hook formula.
	Finally, in Section~\ref{sec:gorensteinsection}, Stanley's theory of  $(P,\omega)$-partitions is applied to show that the $h^*$-vectors of these two polytopes are symmetric and unimodal.

Various conjectures and questions are scattered through the paper.

%%%%%%%%%%%%%%%%%%%%%%%%
%%%%%%%%%%%%%%%%%%%%%%%%

\section{Preliminaries}
\label{sec:pre}

There are a number of concepts to which we refer throughout the paper.
In this section, we collect the most frequent of these notions.

%%%%%%%%%%%%%%%%%%%%%%%%
%%%%%%%%%%%%%%%%%%%%%%%%

\subsection{Diagrams, Wilf Equivalence, and Grid Classes}

Let $\pi = a_1\cdots a_k\in\symm_k$.  Sometimes for clarity we will insert commas and write $\pi=a_1,\cdots,a_k$.
The {\em diagram} of a permutation $\pi$ is the set of points with Cartesian coordinates $(i,a_i)$ for $i = 1,\ldots,k$.
An example diagram is given in Figure~\ref{fig:diagram}.
When no confusion will result, we make no distinction between a permutation and its diagram.
Diagrams of permutations provide an easy way to see how certain permutations can be related geometrically.
For example, the diagrams of $\pi$ and $\pi^{-1}$ are related by reflection across the line $y=x$.
With both the $\Pi$-avoiding permutohedra and $\Pi$-avoiding Birkhoff polytopes,
many results will be true not only for the choice of $\Pi$ in their statement, but also for certain other subsets of permutations whose diagrams are related to those in $\Pi$.

\begin{figure}
	\begin{tikzpicture}
		\draw[step=1cm,very thin] (1,1) grid (7,7);
		    \foreach \i in {(1,2),(2,6),(3,4),(4,1),(5,7),(6,5),(7,3)}
		    {
		    	\fill \i circle [radius=3pt];
		    }
	\end{tikzpicture}
	\caption{The diagram of the permutation $2641753$.}\label{fig:diagram}
\end{figure}
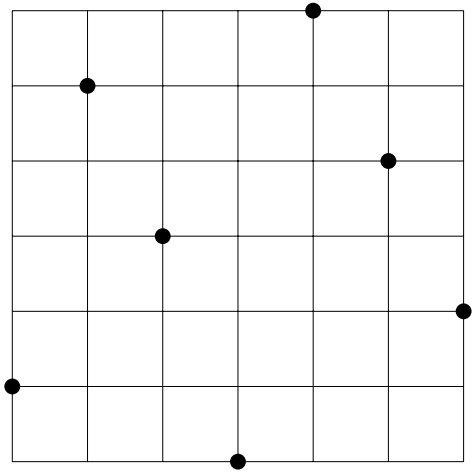

Two permutations $\pi_1$ and $\pi_2$ are called {\em Wilf equivalent}, written $\pi_1 \equiv \pi_2$, if $|\Av_n(\pi_1)| = |\Av_n(\pi_2)|$ for all $n$.
For example,  any two permutations in $\symm_3$ are Wilf equivalent.
This is indeed an equivalence relation. 
Although proving $\pi_1 \equiv \pi_2$ may be quite difficult, in some instances
the Wilf equivalence of two permutations follows quickly from observing that their diagrams are related by a transformation in the dihedral group of the square.

Let $D_4 = \{R_0,R_{90},R_{180},R_{270},r_{-1},r_{0},r_{1},r_{\infty}\}$, where $R_{\theta}$ is rotation counterclockwise by an angle of $\theta$ degrees and $r_m$ is reflection across a line of slope $m$. 
A couple of these rigid motions have easy descriptions in terms of the one-line notation for permutations.
If $\pi = a_1a_2 \ldots a_k$ then its {\em reversal} is $\pi^r =a_k\ldots a_2a_1 =r_{\infty}(\pi)$, and its {\em complement} is $\pi^c = k+1-a_1,\ k+1-a_2,\ \dots,\ k+1-a_k=r_0(\pi)$.

Note that for any $f \in D_4$, one has $\sig \in \Av_n(\pi)$ if and only if $f(\sig) \in \Av_n(f(\pi))$, and hence $\pi \equiv f(\pi)$.
For this reason, the equivalences induced by the dihedral action on a square are often referred to as the {\em trivial Wilf equivalences}. 

Call polytopes $P$ and $Q$ {\em unimodularly equivalent} if one can be taken into the other by an affine transformation whose linear part is representable by an $n \times n$ matrix with integer entries and determinant $\pm1$.  We will see in Propositions~\ref{prop: Pn equivalence} and~\ref{prop:dihedral} that certain trivial Wilf equivalences imply unimodular equivalence of the corresponding polytopes.

In subsequent sections, it will be helpful to describe classes of permutations in the following way:
Let $A = (a_{i,j})$ be a $k \times l$ matrix with entries in $\{0,\pm 1\}$.
We say that a permutation $\sig$ is {\em $A$-griddable} in $\RR^2$ 
if the diagram $\mathcal{C}$ of $\sig$ can be partitioned into rectangular regions $C_{i,j}$ using horizontal and vertical lines in such a way that
\[
	\mathcal{C} \cap C_{i,j} \text{ is } \begin{cases} 
		\text{ increasing } & \text{ if } a_{i,j} = 1, \\
		\text{ decreasing } & \text{ if } a_{i,j} = -1, \\
		\text{ empty } & \text{ if } a_{i,j} = 0.
		\end{cases}
\]
If $\mathcal{C} \cap C_{i,j}$ contains one element or no elements, it may be considered as either increasing or decreasing.
For example, if 
\[
	A = \begin{bmatrix}
			0 & 1 \\
			-1 & -1 \\ 
			0 & -1 
		\end{bmatrix},
\]
then $\sig=4261573$ is $A$-griddable, as demonstrated in Figure~\ref{fig:gridding}.
For a particular matrix $A$, the {\em grid class} of $A$ is the set of permutations that are $A$-griddable.
We will occasionally use grid classes to more conveniently describe the structure of permutations used as the vertices of our polytopes.

\begin{figure}
	\begin{tikzpicture}
		    \foreach \i in {(1,4),(2,2),(3,6),(4,1),(5,5),(6,7),(7,3)}
		    {
		    	\fill \i circle [radius=3pt];
		    }
		\draw (2.5,0.5) -- (2.5,7.5);
		\draw (0.5,1.5) -- (7.5,1.5);
		\draw (0.5,6.5) -- (7.5,6.5);
		\draw (0.5,0.5) -- (0.5,7.5) -- (7.5,7.5) -- (7.5,0.5) -- cycle;
	\end{tikzpicture}
	\caption{An $A$-gridding of $4261573$.}\label{fig:gridding}
\end{figure}
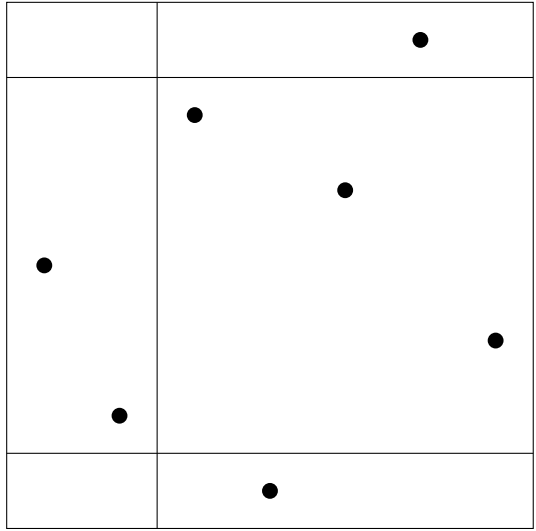

%%%%%%%%%%%%%%%%%%%%%%%%
%%%%%%%%%%%%%%%%%%%%%%%%

\subsection{Ehrhart Polynomials and Volume} 

For a lattice polytope $P \subseteq \RR^n$, consider the counting function $\ehr_P(m) := |mP \cap \ZZ^n|$, where $mP$ is the $m$-th dilate of $P$.
This function is a polynomial in $m$, although not obviously so; it is called the {\em Ehrhart polynomial} of $P$.
In particular, two well-known theorems due to Ehrhart~\cite{Ehrhart} and Stanley~\cite{StanleyDecompositions} imply that the {\em Ehrhart series} of $P$,
\[
	E_{P}(t) := 1 + \sum_{m\ge1} \ehr_P(m)t^m,
\]
may be written in the form
\[
	E_P(t)=\frac{\sum_{j=0}^d h_j^*t^j}{(1-t)^{\dim P+1}}
\]
for some nonnegative integers $h_0^*,\ldots,h_d^*$ with $h_0^*=1$, $h_d^* \neq 0$, and $d \leq \dim P$. 
	
We say the polynomial $h^*_P(t):=\sum_{j=0}^dh_j^*t^j$ is the {\em $h^*$-polynomial} of $P$ and the vector of coefficients, $h^*(P)$, is the \emph{$h^*$-vector} of $P$.
The $h^*$-vector of a lattice polytope $P$ is a fascinating invariant, and obtaining a general understanding of $h^*$-vectors of 
lattice polytopes and their geometric/combinatorial implications is currently of great interest.

A standard result of Ehrhart theory is that the leading coefficient of $\ehr_P(m)$ gives the volume of $P$.
We note, though that when a polytope $P \subseteq \RR^n$ is not full-dimensional, some extra care is needed when discussing volume.
Usual Euclidean volume would dictate that the volume of a polytope that is not full-dimensional is zero.
However, we are typically interested in the {\em relative volume}, that is, the volume of the polytope with respect to the lattice $(\aff P) \cap \ZZ^n$ where $\aff P$ is the affine subspace spanned by $P$.
When $P$ does have full dimension, the notions of volume and relative volume coincide.
Throughout this paper,  ``volume''   is understood to mean the relative volume.

The {\em normalized volume} of a lattice polytope $P \subseteq \RR^n$ is $\Vol P := (\dim P)!\vol(P)$, where $\vol(P)$ is the usual relative volume of $P$.
A lattice simplex $\Sigma \subseteq \RR^n $ with vertex set $V = \{v_0,\ldots, v_k\}$ is {\em unimodular with respect to the lattice $L$} if it has smallest possible relative volume with respect to $L$.
If $L$ is not specified, then it is assumed that $L = (\aff V) \cap \ZZ^n$.
Equivalently, $\Sigma$ is unimodular with respect to $L$ if the set of emanating vectors $\{v_1 - v_0, \ldots, v_k - v_0\}$ forms a $\ZZ$-basis of $L - v_0$. 
In particular, if $P$ is unimodular, then it has a normalized volume of $1$.
We refer to Section~5.4 of~\cite{BeckRobinsCCDed2} for a more thorough discussion of these details.

%%%%%%%%%%%%%%%%%%%%%%%%
%%%%%%%%%%%%%%%%%%%%%%%%

\section{Permutohedra}\label{sec:permutohedra}

The permutohedron has been generalized in multiple ways, including the permuto-associa\-hedron of Kapranov~\cite{Kapranov}, which was first realized as a polytope by Reiner and Ziegler~\cite{ReinerZiegler},
and the generalized permutohedra studied by Postnikov~\cite{Postnikov2009}.
Here, we study yet another generalization of the permutohedron by looking at $P_n$ from the perspective of pattern avoidance.

\begin{defn}
	Let $\Pi \subseteq \symm_n$ and define 
	\[ P_n(\Pi) := \conv\{(a_1,\dots,a_n)\ |\  a_1\dots a_n \in \Av_n(\Pi)\}\]
	to be the {\em $\Pi$-avoiding permutohedron}. 
	If $\Pi=\{\pi\}$ then we write $P_n(\pi)$ for $P_n(\Pi)$.
\end{defn}

Notice that if $\Pi = \emptyset$, then $P_n(\Pi) = P_n$
and each permutation is a vertex of $P_n$.  Since $P_n(\Pi)$ is obtained by taking a convex hull of a subset of these vertices, the elements of $\Av_n(\Pi)$ will also be vertices of $P_n(\Pi)$.
For example, if $\pi \in \symm_3$ then, as previously remarked,
 $|\Av_n(\pi)|=C_n$ where $C_n$ is the $n$th Catalan number, so  $P_n(\pi)$ has a Catalan number of vertices.

\begin{prop}\label{prop: Pn equivalence}
	If $\Pi \subseteq \symm$, then $P_n(f(\Pi))$ is unimodularly equivalent to $P_n(\Pi)$ for any $f \in \{R_0,R_{180},r_0,r_{\infty}\}$.
	So their face lattices, volumes, and Ehrhart series are all equal.
\end{prop}

\begin{proof}
	For ease of notation, we prove this in the case that $\Pi = \{\pi\}$.
	The general demonstration is similar. Recall that $\pi^r=r_\infty(\pi)$ and $\pi^c=r_0(\pi)$.

	From the discussion above, $P_n(\pi^r)$ is the image of $P_n(\pi)$ under the map $f(v) = Av$, where $A = \begin{bmatrix} e_n & \cdots & e_1\end{bmatrix}$ 
	and the $e_i$ are the standard unit column vectors.
	Since $A$ is a permutation matrix, this is a unimodular transformation.
	
	Also, $P_n(\pi^c)$ is the image of $P_n(\pi)$ under the map
	\[
		g(x_1,\ldots,x_n) = (n+1 - x_1, \ldots, n+1-x_n) = (n+1,\ldots,n+1)-(x_1,\ldots,x_n),
	\]
	which is again clearly unimodular. 
	Finally, notice that $R_{180}(\pi)=f \circ g(\pi)$ and so $R_{180}$ gives rise to a unimodular equivalence as well.
\end{proof}

Notice that
\begin{itemize}
	\item two permutations $\pi$ and $\pi^{\prime}$ may be Wilf equivalent without $P_n(\pi)$ and $P_n(\pi^{\prime})$ being unimodularly equivalent.
		For example, $123$ and $132$ are Wilf equivalent, but $P_4(123)$ has $13$ facets whereas $P_4(132)$ has only $11$.
	\item two permutations $\pi$ and $\pi^{\prime}$ may even be trivially Wilf equivalent without $P_n(\pi)$ and $P_n(\pi^{\prime})$ being unimodularly equivalent.
		For example, $\pi = 1423$ and $\pi^{\prime} = 2431$ are related by a $90$-degree rotation, however $P_5(1423)$ has $48$ facets while $P_5(2431)$ only has $46$. 
\end{itemize}

Proposition~\ref{prop: Pn equivalence} allows us to choose $\Pi$ more efficiently; a summary of the choices of $\Pi \subseteq \symm_3$ leading to potentially distinct $P_n(\Pi)$, and the corresponding results, are given in Table~\ref{tab: summary}.
Certain entries in the table have no corresponding result or conjecture provided; this is because no clear structure of $P_n(\Pi)$ is apparent in these cases.
See Table~\ref{tab: Pn data} for experimental data, computed via LattE \cite{latte}, regarding these two polytopes for small $n$.

\begin{table}
\begin{tabular}{|c|c|} \hline
$\Pi$ & Relevant result(s) for $P_n(\Pi)$ \\ \hline
$\emptyset$ & $P_n(\Pi) = P_n$ \\
$\{123\}$ & -- \\
$\{132\}$ & -- \\
$\{123,132\}$ & Theorem~\ref{thm: 123 132} \\
$\{123, 231\}$ & -- \\
$\{123, 321\}$ & $P_n(\Pi) = \emptyset$ for $n \geq 5$ \\
$\{132, 213\}$ & Conjecture~\ref{conj: 132 213} \\
$\{132, 231\}$ & -- \\
$\{132,312\}$ & Proposition~\ref{prop: 132 312} \\
$\{123,132,213\}$ & -- \\
$\{123,132,231\}$ & Proposition~\ref{prop: 123 132 231} \\
%$\{123,213,312\} -- covered
$\{123,132,312\}$ & Proposition~\ref{prop: 123 132 312} \\
%$\{123,213,231\} -- covered
$\{123,231,312\}$ & Proposition~\ref{prop: 123 231 312} \\
$\{132, 213, 231\}$ & Proposition~\ref{prop: 132 213 231} \\
%\{132, 231, 312\} -- covered
%\{132, 213, 312\} -- covered
$\{123,132,213,231\}$ & Proposition~\ref{prop: line segments} \\
%\{123,132,213,312\} -- covered
$\{123,132,231,312\}$ & Proposition~\ref{prop: line segments} \\
%\{123,213,231,312\} -- covered 
$\{132, 213, 231, 312\}$ & Proposition~\ref{prop: line segments} \\
$\{123, 132, 213, 231, 312\}$ & $P_n(\Pi) = \{(n,n-1,\dots,1)\}$ \\ \hline
\end{tabular}
\caption{The choices of $\Pi \subseteq \symm_3$ that result in unimodularly distinct $P_n(\Pi)$, and references to the results proven about them.}\label{tab: summary}
\end{table}

\begin{table}
\begin{tabu}{|c|c|c|c|c|} \hline
$\Pi$ & $n$ & $f_{n-2}$ & $\ehr_{P_n(\Pi)}(m)$ & $\Vol(P_n(\Pi))$ \\ \hline
%	& & & & \\
$\{123\}$ & $3$ & $5$ & $1+\frac{5}{2}m + \frac{5}{2}m^2$ & $5$ \\
	& $4$ & $13$ & $1 + \frac{11}{3}m + 9m^2 + \frac{31}{3}m^3$ & $62$ \\
	& $5$ & $43$ & $1 + \frac{65}{12}m + {121}{8}m^2 + \frac{511}{12}m^3 + \frac{479}{8}m^4$ & $1437$ \\
	& $6$ & $215$ & $1 + \frac{71}{6}m + \frac{117}{4}m^2 + \frac{413}{6}m^3 + \frac{1019}{4}m^4 + \frac{1339}{3}m^5$ & $53560$ \\ \hline
%	& & & & \\ \hline
%	& & & & \\
$\{132\}$ & $3$ & $5$ & $1+\frac{5}{2}m + \frac{5}{2}m^2$ & $5$ \\
	& $4$ & $11$ & $1 + 4m + 9m^2 + 10m^3$ & $60$ \\
	& $5$ & $27$ & $1 + 6m + {37}{2}m^2 + 43m^3 + \frac{109}{2}m^4$ & $1308$ \\
	& $6$ & $84$ & $1 + \frac{521}{60}m + \frac{283}{8}m^2 + \frac{197}{2}m^3 + \frac{2089}{8}m^4 + \frac{22399}{60}m^5$ & $44798$ \\ \hline
%	& & & & \\ \hline
\end{tabu}
\caption{Experimental data for $P_n(123)$ and $P_n(132)$ for $n=3,4,5,6$. We use the notation $f_{n-2}$ to denote the number of facets of the polytope.}\label{tab: Pn data}
\end{table}

%%%%%%%%%%%%%%%%%%%%%%%%
%%%%%%%%%%%%%%%%%%%%%%%%

\subsection{Avoiding Two Patterns in $\symm_3$}

We begin by noting that if $\Pi=\{123,321\}$ then $\Av_n(\Pi)=\emptyset$ for $n\ge5$.  This is because of the
 Erd\H{o}s-Szekeres theorem which states that any permutation in $\symm_{mn+1}$ contains either an increasing
 subsequence of length $m+1$ or a decreasing subsequence of length $n+1$.  The same is clearly true for any $\Pi$
 containing $\{123,321\}$.  So we do not need to consider polytopes for such avoidance classes.

The following result will be useful when considering $\Pi=\{132,312\}$ in both the permutohedron and Birkhoff polytope cases.  It follows easily from the proof of Proposition 5.2 in~\cite{DokosEtAl}.
\begin{lemma}
\label{lem:132,312}
The permutations in $\Av_n(132,312)$ are the permutations of $\symm_n$ in the grid class of the matrix
	\[
		A=\begin{bmatrix}  1\\ -1 \end{bmatrix}.
	\]
\hfill$\qed$
\end{lemma}

\begin{prop}\label{prop:firstclass}
	The polytope $P_n(132,312)$ is a rectangular parallelepiped ({\em parallelotope}). 
	Specifically, the polytope is contained in the hyperplane $\sum x_i = \binom{n+1}{2}$, and its facet-defining inequalities are
	\begin{eqnarray}
		\left| \sum_{i=1}^j (x_i - x_{j+1}) \right| \leq \binom{j+1}{2} \label{eq:pnfacets}
	\end{eqnarray}
	as $j$ ranges over $1,\ldots, n-1$.
\end{prop}

\begin{proof}
	Consider the polytope $P$ defined by the given inequalities and lying in the given hyperplane.
	Each inequality in \eqref{eq:pnfacets} gives a pair of parallel faces of $P$ because of the absolute value signs.
	It is also easy to check that the normal vectors are pairwise orthogonal and also orthogonal to the vector $(1,\ldots,1)$ which defines the hyperplane $\sum x_i = \binom{n+1}{2}$.
	Thus $P$ is an $(n-1)$-dimensional parallelotope.
	
	The polytope $P$ will have $2^{n-1}=|\Av_n(132,312)|$ vertices.  
	So to demonstrate that $P=P_n(132,312)$ it suffices to prove that every $\sig = a_1a_2\cdots a_n\in\Av_n(132,312)$ is a vertex of $P$. 
	It follows from Lemma~\ref{lem:132,312} that the elements of this avoidance class are characterized by the fact that for each $j = 1, \ldots, n-1$, 
	we have $a_{j+1}$ is either one greater than the largest previously-appearing entry or one less than the smallest previously-appearing entry.
	Note that if it is smaller, then $\sig$ satisfies $\sum_{i=1}^j (x_i - x_{j+1}) = \binom{j+1}{2}$, and if
	it is larger then $\sig$ satisfies $\sum_{i=1}^j (x_i - x_{j+1}) = -\binom{j+1}{2}$.
	These equalities hold because the summands are exactly the integers $1,\dots, j$ in the first case and $-1,\ldots,-j$ in the second.
	Since this is true for all $j$, $\sig$ is a vertex of $P$.
\end{proof}

\begin{cor}\label{cor:pnvolume}
	The volume of $P_n(132,312)$ is $(n-1)!$.
\end{cor}

\begin{proof}
	By the previous proposition, the volume of $P = P_n(132,312)$ may be computed directly by choosing a base vertex, taking the product of the lengths of the edges incident to it, and then dividing by an appropriate factor to account for the relative volume.
	For the scaling factor, it is well-known that for a (measurable) subset $S \subseteq \RR^m$ and a linear function $f: \RR^m \to \RR^n$, with $m \leq n$,
	\[
		\vol(f(S)) = \sqrt{\det A^TA}\vol(S),
	\]
	where $A$ is the matrix for $f$ and volume is taken with respect to the usual Euclidean measure.
	In our case, a $\ZZ$-basis for $\aff P \cap \ZZ^n$ is $e_1 - e_j$ for $j = 2,\ldots,n$, so these vectors form the columns of $A$.
	It is straightforward to check that $A^TA = J_{n-1} + I_{n-1}$ where $J_{n-1}$ is the $(n-1)\times (n-1)$ matrix with every entry $1$.
	Furthermore,   one easily sees that $ J_{n-1} + I_{n-1}$ has one eigenvalue equal to $n$ (with corresponding eigenspace spanned by the all-ones vector) and the rest equal to $1$ (with corresponding eigenspace the subspace of vectors with coordinate sum zero).  Thus $\det A^TA =n$.
So to find the relative volume of $P$, we must divide the usual $(n-1)$-dimensional volume of $P$ by $\sqrt{n}$.

	Now, a convenient choice of base vertex is the permutation $\sig = 12\cdots n$.
	Using the hyperplane description of the previous result, this vertex is adjacent to the permutations 
$\sig_j = 2, \cdots, j, 1, j+1,\cdots, n$ for each $j = 2,\ldots, n$.
	It is straightforward to compute that $|\sig_j - \sig| = \sqrt{j(j-1)}$, so taking the product of these lengths and then dividing by $\sqrt{n}$ yields $\vol(P) = (n-1)!$ as desired.
\end{proof}

\begin{remark}
	We would like to note a connection between permutations avoiding $\{132,312\}$ and the world of polytopes.
	The permutations of $\Av_n(132,312)$ can be considered as elements of a type-$A$ Coxeter group.
	Thought of in this way, the elements of $\Av_n(132,312)$ are an example of \emph{$c$-singletons} (where $c=s_1s_2s_3$), that is, their inverses form vertices of both the permutohedron and Loday's realization of the associahedron; see \cite{Loday, HohlwegSingletons}.
	It would be interesting to define pattern-avoiding polytopes for other Coxeter groups and see if there is any relationship with the corresponding $c$-singletons.
\end{remark}

Postnikov~\cite{Postnikov2009} defined generalized permutohedra and showed that they encompass associahedra, cyclohedra, Stanley-Pitman polytopes, and graphical zonotopes.  So one could ask if $P_n(\Pi)$ is always a generalized permutohedron, since we would then immediately know its volume and, in some cases, its Ehrhart polynomial.  However we will show that this is not the case for $\Pi=\{132,312\}$.  To do this, we need a few more tools.

A \emph{fan} in $\RR^n$ consists of a set of polyhedral cones $\FF = \{C_{\alpha}\}$ in $\RR^n$, each containing $0$, such that 
\begin{itemize}
	\item if $C_{\alpha} \in \FF$ and $C_{\beta}$ is a face of $C_{\alpha}$, then $C_{\beta} \in \FF$, and
	\item for any $\alpha$ and $\beta$, $C_{\alpha} \cap C_{\beta}$ is a face of both $C_{\alpha}$ and $C_{\beta}$. 
\end{itemize}
Using the notation 
\[
	|\FF| := \bigcup_{F \in \FF} F,
\]
we say a fan $\FF^{\prime}$ {\em refines} $\FF$ if $|\FF^{\prime}| = |\FF|$ and if each cone in $\FF^{\prime}$ is contained in a cone in $\FF$.
We note that the  literature also uses the notation $\bigcup \FF$ for $|\FF|$.

Let $w \in \RR^n$ and let $P \subseteq \RR^n$ be any polytope.
Define
\[
	\face_w(P) := \{ u \in P\ |\ w^Tu \geq w^Tv \text{ for all } v \in P\}.
\]
In other words, $\face_w(P)$ is the face of $P$ for which the linear form defined by $w$ is maximized.
If $F$ is a face of a polytope $P$, the {\em normal cone} of $F$ at $P$ is
\[
	N_P(F) := \{w \in \RR^n\ |\ \face_w(P) = F\}.
\]
In particular, if $F$ is a facet of $P$, then $N_P(F)$ is a ray. 
The collection of all $N_P(F)$, ranging over all faces of $P$, is the {\em normal fan} of the polytope, and is denoted $N(P)$.

In our case, the inequalities of $\eqref{eq:pnfacets}$ provide the rays of the normal fan for $P_n(132,312)$.
We will compare this normal fan with a certain other fan, defined in the following way.
The {\em braid arrangement} in $\RR^n/(1,\ldots,1)\RR$ is the set of hyperplanes $\{x_i = x_j\}_{1 \leq i < j \leq n}$.
These hyperplanes partition the space into the {\em Weyl chambers}
\[
	C_{\sig} := \{(x_1,\dots,x_n) \in \RR^n \mid x_{\sig(1)} \leq \cdots \leq x_{\sig(n)}\},
\]
where $\sig \in \symm_n$.
The collection of these chambers and their lower-dimensional faces is the {\em braid arrangement fan}.
The following result of Postnikov, Reiner, and Williams, allows us to see that $P_n(132,312)$ does not fall into the class of generalized permutohedra.

\begin{prop}[{\cite[Proposition 3.2]{PostnikovEtAl}}]
	A polytope $P$ in $\RR^n$ is a generalized permutohedron if and only if its normal fan, reduced by $(1,\ldots,1)\RR$, is refined by the braid arrangement fan. \qed
\end{prop}

Using the hyperplane description from Proposition~\ref{prop:firstclass}, we can see immediately that the rays of $N(P_n(132,312))$ are not all rays of the braid arrangement fan.
Thus, the braid arrangement fan cannot be a refinement of $N(P_n(132,312))$.
See Figure~\ref{fig:braids} for an example.

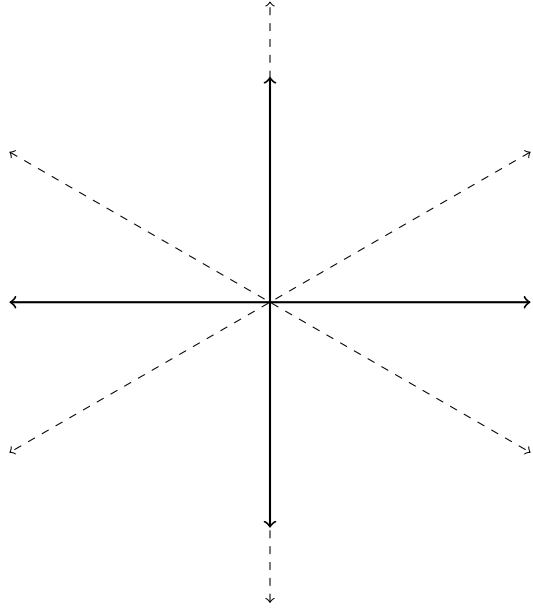
\begin{figure}
	\begin{tikzpicture}
	\draw[dashed, <->] (0,-4) -- (0,4);
	\draw[dashed, <->] (-3.464,2) -- (3.464,-2);
	\draw[dashed, <->] (3.464,2) -- (-3.464,-2);
	
	\draw[thick,<->] (0,-3) -- (0,3);
	\draw[thick,<->] (-3.464,0) -- (3.464,0);
	
	\end{tikzpicture}
	\caption{Viewed from $(1,1,1)$, the rays $N(P_3(132,312))$ are solid, while the rays of the braid arrangement fan are dashed.}\label{fig:braids}
\end{figure}

The Ehrhart polynomial of $P_n$ is known to be $\sum_{i=0}^{n-1} F_i m^i$, where $F_i$ is the number of forests with $i$ edges on vertex set $\{1,2,\dots,n\}$ (see Exercise 4.64(a) in~\cite{StanleyVol1Ed2}).
The technique in this exercise can also be used to find the Ehrhart polynomial of $P_n(\{132,312\})$.
Our first step in this direction will use the following result, due to Stanley.
	
\begin{thm}[{\cite[Theorem 2.2]{StanleyZonotope}}]
\label{zonotope}
	Suppose $P$ is a lattice zonotope, that is, $P$ can be written in the form
	\[
		P = \{a_1v_1 + \ldots + a_kv_k\ |\ 0 \leq a_i \leq 1\},
	\]
	where each $v_i$ belongs to $\ZZ^n$.
	The Ehrhart polynomial of $P$ is
	\begin{equation} \label{g(X)}
		\ehr_P(m) = \sum_X g(X)m^{|X|}
	\end{equation}
	where the sum ranges over all linearly independent subsets $X$ of $\{v_1,\ldots, v_k\}$ 
	and where $g(X)$ is the greatest common divisor of all full minors of the matrix whose columns are the elements of $X$. \qed
\end{thm}

To state the next result elegantly we define, for nonnegative integers $n$ and $k$, the {\em falling factorial}
\[
n\da_k = n(n-1)\dots(n-k+1).
\]

\begin{prop}\label{prop:ehrhartpoly}\label{prop: 132 312}
	The polytope $P = P_n(132,312)$ has Ehrhart polynomial
	\[
		\ehr_P(m) = \sum_{k=0}^{n-1} (n-1)\da_k m^k.
	\]
\end{prop}

\begin{proof}
	From the half-space and hyperplane description given in Proposition~\ref{prop:firstclass}, we can see that $P$ is, up to a translation by $(1,2,\ldots,n)$, the zonotope
	\[
		Z = \{a_1v_1 + \ldots + a_{n-1}v_{n-1}\ |\ 0 \leq a_i \leq 1\} \subseteq \RR^n
	\]
	where $v_j = \sum_{i=1}^j (e_i - e_{j+1})$ for $j = 1,\ldots,n-1$.
	By applying the transformation $x \mapsto Ax$, where $A$ is the $n\times n$ upper-triangular matrix with $-1$ in all positions along and above the diagonal,
	we see that $Z$ is unimodularly equivalent to
	\[
		\hat{Z} = \{a_1w_1 + \ldots + a_{n-1}w_{n-1}\ |\ 0 \leq a_i \leq 1\} \subseteq \RR^n
	\]
	where 
	\[
		w_j = \sum_{i=1}^{j+1} (i-1)e_i
	\]
	for each $j = 1, \ldots, n-1$.
	Note that the set of all $w_j$ is linearly independent.
	
We will now complete the proof using equation~(\ref{g(X)}) on the $w_j$ basis. 
First, however, we need to set up some notation.  
For $X$ as in~(\ref{g(X)}) we will use $X$ to stand for both the subset and the matrix whose columns are the elements of $X$. 
For any family $\cF$ of subsets $X$ we  define
\[
	g(\cF)=\sum_{X\in\cF} g(X).
\]
We also let $\cF_{n,k}$ be the family of all $k$-element subsets of $w_1,\dots,w_{n-1}$ and $g(n,k)=g(\cF_{n,k})$. 
So we will be done if we can  prove that $g(n,k)=(n-1)\da_k$.  In fact, we will show that the following recurrence relation holds:
\begin{equation}
\label{grec}
g(n+1,k)=g(n,k)+n(g(n,k-1)-g(n-1,k-1))+g(n-1,k-1).
\end{equation}
It is easy to verify that $(n-1)\da_k$ satisfies the same recursion for $n\ge2$.  So  induction on $n$ completes the proof once we have verified the base case $n=1$.  But $P=P_1(132,312)$ is a single vertex so that $\ehr_P(m)=1$ which agrees with the fact that $0\da_k =\delta_{0,k}$ where the latter is the Kronecker delta.

To prove \eqref{grec}, partition $\cF_{n+1,k}$ into the three subsets
\begin{align*}
	\cF_1&=\{X\in\cF_{n+1,k}\ |\ \text{$X$ does not contain $w_n$}\},\\
	\cF_2&= \{X\in\cF_{n+1,k}\ |\ \text{$X$ contains both $w_{n-1}$ and $w_n$}\},\\
	\cF_3&= \{X\in\cF_{n+1,k}\ |\ \text{$X$ contains $w_n$ but not $w_{n-1}$}\}.
\end{align*}
From the definitions, one has
\[
	g(n+1,k)=g(\cF_1)+g(\cF_2)+g(\cF_3).
\]
We now show that each of these summands equals the corresponding summand in~(\ref{grec}).

The matrices in $\cF_1$ are the same as those for $\cF_{n,k}$ except with a last row of zeros. 
Clearly this row does not contribute any nonzero minors so $g(\cF_1)=g(n,k)$, giving the first summand.

Now consider the minors of a matrix $X\in\cF_2$, letting $M$  be the submatrix of the minor. 
An example follows the proof to elucidate the method.
If $M$ does not contain the last row of $X$, then its last two columns are equal and $\det M = 0$. 
So the only $M$ contributing to $g(\cF_2)$ are those whose last row is the final row of $X$ which is all zero except for a last entry of $n$.  
It follows that $|\det M| = n|\det M'|$ where $M'$ is obtained by removing the last row and column of $M$. 
The possible $M'$ which can appear are exactly those occurring in elements $X'\in\cF_{n,k-1}$ such that $w_{n-1}\in X'$. 
Using the reasoning of the previous paragraph and complementation, we see that such $|\det M'|$ contribute exactly $g(n,k-1)-g(n-1,k-1)$ to the desired sum.  
Thus $g(\cF_2)=n(g(n,k-1)-g(n-1,k-1))$.

Finally take $X\in\cF_3$ so that $X$ ends with a sequence of at least two rows each of which has a sole nonzero entry at the end. 
Keeping the notation and reasoning of the previous paragraph, we see that if $\det M\neq 0$ then $M$ must contain exactly one row from this final sequence. 
Let $m_1,\dots,m_r$ be the minors which can be obtained from all nonzero minors containing the last row of $X$. 
Then for all $i$ we have $m_i=n m_i'$ where $m_1',\dots,m_r'$ are exactly the nonzero minors of $X'\in\cF_{n-1,k-1}$ obtained by removing the last row and column of $X$. 
So 
\[
	\gcd(m_1,\dots,m_r)=n\gcd(m_1',\dots,m_r')=n g(X').
\]  
Now repeat this process, but using the penultimate row of $X$, giving minors $m_{r+1},\dots,m_{2r}$ with greatest common divisor $(n-1)g(X')$. 
But $n$ and $n-1$ are relatively prime, so $\gcd(m_1,\dots,m_{2r})=g(X')$. 
Continuing in this way, we see that $g(X)=g(X')$. 
Summing over all possible $X$ gives $g(\cF_3)=g(n-1,k-1)$ and completes the proof.
\end{proof}

To illustrate this demonstration, take $n=5$ and $k=4$ .  Then a typical element of $\cF_2$ is
$$
X=\left[\begin{array}{cccc}
0&0&0&0\\
1&1&1&1\\
0&2&2&2\\
0&0&3&3\\
0&0&4&4\\
0&0&0&5
\end{array}\right].
$$
Considering the submatrix $M$ obtained by picking rows $2$, $3$, $5$, and $6$ of $X$ and expanding around the last row we get $\det M = 5\det M'$ where
$$
M'=\left[\begin{array}{ccc}
1&1&1\\
0&2&2\\
0&0&4
\end{array}\right].
$$
Note that $M'$ is also a submatrix of the matrix
$$
X'=\left[\begin{array}{ccc}
0&0&0\\
1&1&1\\
0&2&2\\
0&0&3\\
0&0&4\\
\end{array}\right]\in\cF_{5,3}
$$
and $w_4\in X'$.  The reader should now find it easy to construct a similar example for the argument concerning $X\in\cF_3$ if need be.

A standard fact from Ehrhart theory states that the leading coefficient of $\ehr_P(m)$ is the volume of $P$, so Corollary~\ref{cor:pnvolume} is reaffirmed by the previous result.
Moreover, knowing the Ehrhart polynomial allows us to deduce an interesting fact about the interior lattice points of $P_n(132,312)$.

\begin{cor}
	The number of lattice points interior to $P_n(132,312)$ is equal to the number of derangements in $\symm_{n-1}$.
\end{cor}

\begin{proof}
	Let $P = P_n(132,312)$ and $P^{\circ}$ be the interior of $P$. By Proposition~\ref{prop:ehrhartpoly} and Ehrhart-Macdonald reciprocity \cite[Theorem~4.1]{BeckRobinsCCDed2},
	\[
		\ehr_{P^{\circ}}(m) = (-1)^{n-1}\sum_{k=0}^{n-1} (n-1)\da_k (-m)^k.
	\]
	Evaluating at $m=1$, we get
	\[
		\ehr_{P^{\circ}}(1) = (-1)^{n-1}\sum_{k=0}^{n-1} (n-1)\da_k (-1)^k,
	\]
	which is the well-known inclusion-exclusion formula for derangements.
\end{proof}

\begin{question}
	Is there a natural bijection between the interior points of $P_n(132,312)$ and the derangements in $\symm_{n-1}$?
\end{question}

In the case of $P_n(132,312)$, the Ehrhart polynomial was simple enough to compute directly. 
Since the coefficients can be explicitly determined, one may also determine the $h^*$-vector of $P_n(132,312)$ by a change-of-basis, although there does not seem to be a simple formula for its components.

Although finding explicit formulas for $h^*$-vectors is usually challenging in general, there are other methods for determining certain properties it might possess.
A recent result due to Beck, Jochemko, and McCullough~\cite{BJMminkowskisum} states that lattice zonotopes always have a unimodal $h^*$-vector.
Thus the following result follows from Proposition~\ref{prop:firstclass}.

\begin{cor}
	For all $n \geq 1$, $h^*(P_n(132,312))$ is unimodal. \qed
\end{cor}

\begin{question}
	For which $\Pi$-avoiding permutohedra $P$ is $h^*(P)$ unimodal?
\end{question}

We will next consider a $\Pi$-avoiding permutohedron whose Ehrhart polynomial is easily computable due to results of Pitman and Stanley~\cite{PitmanStanley}.
Given a sequence of nonnegative real numbers $c = (c_1,\ldots,c_n)$, there is a corresponding {\em Pitman-Stanley polytope} $PS_n(c)$ defined by
\[
	PS_n(c) := \left\{x \in \RR^n\ |\ x_i \geq 0 \text{ and } \sum_{i=1}^jx_i \leq \sum_{i=1}^jc_i \text{ for all } 1 \leq j \leq n\right\}.
\]
Pitman-Stanley polytopes are connected with multiple combinatorial objects.
For example, recall that a {\em polyhedral subdivision} of a polytope $P$ is a collection of subpolytopes $P_1,\ldots,P_k \subseteq P$ whose union is $P$, and $P_i \cap P_j$ is a face of both $P_i$ and $P_j$ for all $i,j$. Pitman and Stanley showed that $PS_n(c)$ has polyhedral subdivisions whose maximal elements of correspond to certain plane trees; $\Vol(PS_n(c))$ can be expressed in terms of parking functions; the number of lattice points of $PS_n(c)$ can be expressed in terms of plane partitions of a particular shape.
The key result for us is the following.

\begin{thm}[Pitman and Stanley, \cite{PitmanStanley}]\label{thm:pitmanstanleyehrhart}
	Let $a,b$ be positive integers, and set $c = (a,b,\ldots,b) \in \ZZ^n$.
	The Ehrhart polynomial of $PS_n(c)$ is
	\[
		\ehr_{PS_n(c)}(m) = \frac{am + 1}{n!}\prod_{j=2}^n \left((a + nb)m + j\right).
	\] \qed
\end{thm}

Before continuing, we need a little background. 
The {\em face lattice} of a polytope is the poset of its faces ordered by inclusion.
Two polytopes are {\em combinatorially equivalent} if their face lattices are isomorphic.
As proven in Theorem 19 of~\cite{PitmanStanley}, whenever $c$ has positive entries, $PS_n(c)$ is combinatorially equivalent to an $n$-cube.

\begin{lemma}\label{lem:PS vertices}
	When $c$ has positive entries, the vertices of $PS_n(c)$ are exactly the vectors $v = (v_1,\ldots,v_n)$ constructed, component-wise from left to right, by either setting $v_j = 0$ or setting $v_j = c_j+c_{j-1}+\dots+c_i$,
	where $v_{i-1}$ is the previous nonzero entry of $v$.
\end{lemma}

\begin{proof}
	Since $c$ has positive entries, $PS_n(c)$ is a combinatorial cube, hence the set of facets may be partitioned into $n$ non-intersecting pairs.
	In particular, the pairs correspond to the hyperplanes $x_j = 0$ and $x_1 + \cdots + x_j = c_1 + \cdots + c_j$.
	Again, since $PS_n(c)$ is a combinatorial cube, a vertex $v$ will lie on exactly one of the facets of each pair.
	From these two facts, the conclusion follows.
\end{proof}

\begin{thm}
\label{thm: 123 132}
	The polytope $P=P_n(123,132)$ is a combinatorial cube with Ehrhart polynomial
	\[
		\ehr_{P}(m) =\frac{m + 1}{(n-1)!}\prod_{j=2}^{n-1}{(nm + j)}.
	\]
\end{thm}

\begin{proof}
We will show that $P$ is related to $PS_{n-1}(1,\dots,1)$ in such a way that its face lattice and Ehrhart polynomial are preserved.  Then the theorem will follow from the statement just before Lemma~\ref{lem:PS vertices}, and by setting $a=b=1$ in Theorem~\ref{thm:pitmanstanleyehrhart}.
	
We first need a description of the vertices of $P$.
By reversing the permutations in Proposition 4.2 of~\cite{DokosEtAl}, we note that the diagram for a vertex $v = (v_1,\dots,v_n)$ of $P_n(123,132)$ consists of a decreasing sequence of blocks where each block is the pattern $k(k-1)\cdots 1(k+1)$ for some $k$.
	Define a function $f: \RR^n \to \RR^{n-1}$ by
	\[
		f(a_1,\ldots,a_n) = (a_1,\ldots,a_{n-1}) - (n-1,n-2,\ldots,1). 
	\]
We claim that $f$  maps the vertices of $P$ to the vertices of $PS_{n-1}(1,\ldots,1)$. 
Indeed, suppose the first block of a vertex $v$ of $P$ is of the form $(n-1,n-2,\dots,n-k,n)$.  Then under $f$ this maps to the sequence 
$(0,0,\dots,0,k+1)$ with $k$ initial zeros.  But, by Lemma~\ref{lem:PS vertices}, this is the prefix of a vertex of 
$PS_{n-1}(1,\dots,1)$.  Continuing in this way, we see that $f(v)$ will indeed be a vertex of this Pitman-Stanley polytope. 
Reversing the argument shows that $f$ is, in fact, a bijection on the vertex sets.
	
Since $P$ is a subpolytope of the usual permutohedron, the projection to the first $n-1$ coordinates preserves the face lattice and Ehrhart polynomial, as does lattice translation.
This verifies the claim in the first sentence of the proof.
\end{proof}

From the Ehrhart polynomial, we can immediately determine the volume and number of lattice points in the polytope.

\begin{cor}\label{cor: 123 132 corollary}
	The normalized volume of $P_n(123,132)$ is $n^{n-2}$ and the number of lattice points it contains is the Catalan number 
\[
	C_n=\frac{1}{n+1}\binom{2n}{n}.
\]
\end{cor}

\begin{proof}
	To calculate the normalized volume, one takes the leading coefficient of the Ehrhart polynomial in Theorem~\ref{thm: 123 132} and multiplies by $(n-1)!$ since $\dim P_n(123,132)=n-1$. 
	To calculate the number of lattice points, one just plugs $m=1$ into this polynomial.
\end{proof}

We end this section with a question and a conjecture.

\begin{question}
	The normalized volume in	Corollary~\ref{cor: 123 132 corollary} is just the number of trees on $n$ vertices and this quantity will also appear as a normalized volume in Proposition~\ref{prop: 123 231 312}.  
	And there are many combinatorial interpretations of the Catalan numbers.  
	This raises the question of whether there is a combinatorial proof of Corollary~\ref{cor: 123 132 corollary} or Proposition~\ref{prop: 123 231 312}.
\end{question}

The conjecture that follows makes a statement similar to that of Proposition~\ref{thm: 123 132}.
However, we have been unable to provide a proof.

\begin{conj}\label{conj: 132 213}
	For all $n$, $P_n(132,213)$ is a combinatorial cube with normalized volume $2^{n-1}n^{n-3}$.
\end{conj}

%%%%%%%%%%%%%%%%%%%%%%%%
%%%%%%%%%%%%%%%%%%%%%%%%

\subsection{Avoiding Three or Four Patterns from $\symm_3$}

When $\Pi$ contains at least three or four patterns of $\symm_3$, there are relatively few vertices of $P_n(\Pi)$.
Consequently, $P_n(\Pi)$ can be a farily simple object such as a simplex or line segment.

\begin{prop}\label{prop: 123 132 312}
	The Ehrhart polynomial for $P = P_n(123,132,312)$ is $(1+m)^{n-1}$ and so $h^*_P(t)$ is the Eulerian polynomial $A_{n-1}(t)$.
\end{prop}

\begin{proof}
	As noted in~\cite{BraunUnimodality}, it is implied by~\cite{CLS2005} that the simplex $P_n^{\prime}$ whose vertices are the set
	\[
		L_n := \{e_n\} \cup \left\{ \sum_{j=i}^n je_j \mid i = 1, \dots, n-1 \right\}
	\]
	has Ehrhart polynomial $(1+m)^{n-1}$.
	Since the degree of the Ehrhart polynomial is the dimension of the polytope, $P_n^{\prime}$ is an $(n-1)$-dimensional simplex.
	In particular, note that each $(x_1,\dots,x_n) \in L_n$  satisfies the equation $x_n - x_{n-1} = 1$.
	So, projecting $P_n^{\prime}$ to $\RR^{n-1}$ by forgetting the last coordinate one obtains $P_n^{\dprime}$, which has the same Ehrhart polynomial as $P_n^{\prime}$.
	Transforming $P_n^{\dprime}$ by $f: x \mapsto Ax$, where $A$ is the matrix with $j$th column $e_j - e_{j+1}$ for $j = 1,\ldots,n-2$ and last column $e_{n-1}$, results in the simplex whose vertices are $0$ and
	$ie_i + \sum_{j=i+1}^{n-1} e_j$ for $i=1,\ldots, n-1$.
	
	As stated in the proof Proposition 16* from the paper of Simion and Schmidt~\cite{SimionSchmidt}, the $n$ permutations in $\Av_n(123,132,312)$ are those obtained by inserting $n$ in all possible ways (between  elements or at the beginning or end) into the decreasing sequence $n-1, n-2, \dots, 1$.
	So $f(P_n^{\dprime})$ can also be obtained from $P$ by dropping the last coordinate and translating by $-(n-1,n-2,\ldots,1)$.
	Since each of these operations is a unimodular transformation, $P$ has the same Ehrhart polynomial and $h^*$-polynomial as $P_n^{\prime}$, which are $(1+m)^{n-1}$ and $A_{n-1}(t)$, respectively.
\end{proof}

Recall that the $(n-1)$-dimensional {\em standard simplex} is the simplex $\Delta_{n-1} \subseteq \RR^n$ whose vertices are the standard basis vectors of $\RR^n$.

\begin{prop} \label{prop: 123 132 231}
	For all $n$, $P_n(123,132,231)$ is unimodularly equivalent to $\Delta_{n-1}$.
\end{prop}

\begin{proof}
Again from the proof of \cite[Proposition 16*]{SimionSchmidt} we see that the elements of $\Av_n(123,132,231)$ are exactly the permutations of the form
\[
	\sig = n, n-1, \dots, k+1, k-1, k-2, \dots, 1, k
\]
for $1\le k\le n$.
Consider the transformation $f: \RR^n \to \RR^n$, defined by $f(x) = Ax - v$, where column $i$ of $A$ is $e_{n-i} - e_{n-i+1} + e_n$ for $1 \leq i < n$, column $n$ of $A$ is $e_n$, and $v = (1,\dots,1,\binom{n}{2})^{T}.$
It is straightforward to check that $\det A = (-1)^{n-1}$, so that $f$ is a unimodular transformation, and that $f(P_n(123,132,231)) = \Delta_{n-1}$.
\end{proof}

\begin{prop}\label{prop: 123 231 312}
	For all $n$, $P_n(123,231,312)$ is a simplex with normalized volume $n^{n-2}$.
\end{prop}

\begin{proof}
Using the proof of Proposition 16* in~\cite{SimionSchmidt} again, the elements of $\Av_n(123,231,312)$ are
\[
	\sig= k, k-1,\dots, 1, n, n-1, \dots, k+1
\]
for $1\le k\le n$.
Consider the transformation $f: \RR^n \to \RR^n$, defined by $f(x) = Ax + v$, where the first column of $A$ is $-e_1-e_n$, column $i$ of $A$ is $e_{i-1} - e_i$ for $1 < i < n$, column $n$ of $A$ is $e_{n-1}$, and $v = (1,\dots,1,n).$
It is easy to see that $\det A = (-1)^n$, so $f$ is a unimodular transformation.
Moreover, 
\[
	P' = f(P_n(123,231,312)) = \conv(\{0\}\cup\{ne_i + (n-i)e_n \mid 1 \leq i < n\}).
\]
So, $P_n(123,231,312)$ is a simplex lying in the hyperplane $H$ determined by the equation
\[
	(n-1)x_1 + (n-2)x_2 + \dots + x_{n-1} - nx_n = 0.
\]
The lattice $H \cap \ZZ^n$ has a $\ZZ$-basis 
\[
	\{e_i - (n-i)e_{n-1} \mid 1 \leq i \le n-2\} \cup \{ne_{n-1} + e_n\}.
\]
Therefore, the normalized volume of $P'$ is the same as the normalized volume of the polytope whose vertices are 
the coordinate vectors of the vertices of $P'$ expressed in this basis.
These vertices are $0$, $e_{n-1}$, and $ne_i + (n-i)e_{n-1}$ for $1 \leq i \leq n-2$.
Thus, the normalized volume is
\[
	\det [ne_1 + (n-1)e_{n-1},\  ne_2 + (n-2)e_{n-1},\ \dots,\  ne_{n-2} + 2e_{n-1},\  e_{n-1}] = n^{n-2},
\]
as desired.
\end{proof}

\begin{prop}\label{prop: 132 213 231}
	For all $n$, $P_n(132, 213, 231)$ is a simplex with normalized volume $(n-1)!$.
\end{prop}

\begin{proof}
Again,  the proof of \cite[Proposition 16*]{SimionSchmidt} shows that the elements of $\Av_n(132,213,231)$ are exactly the permutations of the form
\[
	\sig = n, n-1, \dots, k+1, 1, 2, \dots, k-1, k
\]
for $1\le k\le n$.
Consider the transformation $f: \RR^n \to \RR^n$, defined by $f(x) = Ax$, where row $1$ of $A$ is the all-ones vector, row $2$ of $A$ is $e_n$, and row $i$ for $2 < i \leq n$ is $e_{i-1} - e_i$.
It is routine to verify (say, by cofactor expansion along row $2$) that $\det A = (-1)^n$, hence $f$ is a unimodular transformation.
Also, it is straightforward to check that $f(P_n(132,213,231))$ lies on the hyperplane $x_1 = \binom{n+1}{2}$, so that projecting $f(P_n(132,213,231))$ onto its last $n-1$ coordinates results in a polytope $P'_n$ with the same normalized volume.
Furthermore, the translation $P'_n + v$, where $v = (-n,1,\dots,1)^T \in \RR^{n-1}$, gives a polytope whose vertices are $v_0,\dots,v_{n-1}$, where $v_0$ is the origin, $v_1 = -e_1$, and
\[
	v_i = -ie_1 + (n-i+1)e_i + 2\sum_{j = 2}^{i-1} e_j
\]
for $i = 2,\dots,n-1$.
The matrix $M$ with  columns  $v_1,\dots,v_{n-1}$  is diagonal with $|\det M| = (n-1)!$ which simultaneously proves that $P_n(132,213,231)$ is a simplex and has the correct normalized volume.
\end{proof}

Although we have proven that $P_n(132,213,231)$ has the same combinatorial structure and volume as $P_n(123,132,312)$, these two polytope are not unimodularly equivalent.
This can be seen from comparing their Ehrhart polynomials, which are distinct; some of the Ehrhart polynomials of $P_n(132,213,231)$ for small $n$ are given in Table~\ref{tab: 132 213 231}.
There does not appear to by any obvious formula for their coefficients.

\begin{table}
\begin{tabu}{|c|c|} \hline
$n$ & Ehrhart polynomial of $P_n(132,213,231)$ \\ \hline
$3$ & $1 + 2m + m^2$ \\
$4$ & $1 + 2m + 2m^2 + m^3$ \\
$5$ & $1 + \frac{7}{3}m + 3m^2 + \frac{8}{3}m^3 + m^4$ \\
$6$ & $1+3m + \frac{9}{2}m^2 + \frac{9}{2}m^3 + 3m^4 + m^5$ \\
$7$ & $1 + \frac{46}{15}m + \frac{31}{6}m^2 + \frac{41}{6}m^3 + \frac{19}{3}m^4 + \frac{18}{5}m^5 + m^6$ \\
$8$ & $1+\frac{91}{30}m + \frac{19}{3}m^2 + \frac{125}{12}m^3 + \frac{35}{3}m^4 + \frac{171}{20}m^5 + 4m^6 + m^7$ \\ \hline
\end{tabu}
\caption{The Ehrhart polynomial of $P_n(132,213,231)$ for $3 \leq n \leq 8$.}\label{tab: 132 213 231}
\end{table}

The last result of this section is included for completeness.

\begin{prop}\label{prop: line segments}
	For all $n$, the polytopes $P_n(123, 132,213,231)$, $P_n(123, 132, 231, 312)$, and $P_n(132, 213, 231, 312)$ are line segments.
\end{prop}

\begin{proof}
	By~\cite[Proposition 17]{SimionSchmidt}, $|\Av_n(\Pi)|=2$ for each of these $\Pi$.
	The claim follows immediately.
\end{proof}

%%%%%%%%%%%%%%%%%%%%%%%%
%%%%%%%%%%%%%%%%%%%%%%%%

\section{The Birkhoff Polytope}\label{sec:birkhoff}

We come now to our second mixing of polytopes and avoidance classes of permutations by generalizing the Birkhoff polytope $B_n$ in the following way.

\begin{defn}
	Let $\Pi$ be any set of permutations.
	The {\em $\Pi$-avoiding Birkhoff polytope} is
	\[
		B_n(\Pi) := \conv\{M \in \RR^{n\times n}\ |\ M \text{ is the permutation matrix for some } \sig \in \Av_n(\Pi)\}.
	\]
\end{defn}

Despite its simple description, the Birkhoff polytope has shown a  reluctance to provide researchers with information about certain elements of its structure.
For example, although its $h^*$-vector is known to be symmetric and unimodal~\cite{AthanasiadisBirkhoff}, its volume is only known for $n \leq 10$~\cite{beckpixton}.

Studying variations of the Birkhoff polytope is not uncommon.
For example, {\em permutation polytopes}, subpolytopes of $B_n$ whose vertices form a subgroup of $\symm_n$,
have been studied by, for example, Burggraf, De Loera, and Omar~\cite{DeLoeraOmarBurggraf}, who studied their volumes, and Onn~\cite{Onn}, who studied their low-dimensional skeletons and combinatorial types.
Another important variation is the class of \emph{transportation polytopes}, in which row and column sums may be numbers other than $1$, and two rows or columns do not necessarily need to sum to the same value.
See \cite{DeLoeraKim} for a nice survey of these polytopes.

For compatibility with diagrams of permutations, we will henceforth use the nonstandard convention of indexing our matrices using Cartesian coordinates, using the convention for permutation diagrams.  So if $M=(m_{x,y})$ is a matrix then $m_{x,y}$ refers to the entry which is in the $x$th column from the left and $y$th row from the bottom.   By way of illustration, in a $3\times3$ matrix we would have
\[
	M=\begin{bmatrix}
		m_{1,3} & m_{2,3} & m_{3,3}\\
		m_{1,2} & m_{2,2} & m_{3,2}\\
		m_{1,1} & m_{2,1} & m_{3,1}
	\end{bmatrix}
\]

If $\sigma \in \symm_n$ is a permutation and we refer to its matrix, we mean the permutation matrix $(m_{x,y}) \in \RR^{n \times n}$ such that $m_{x,y} = 1$ if and only if $(x,y)$ is in the diagram of $\sigma$.
For example, if $\sigma=132$ then the corresponding matrix is
\[
	\begin{bmatrix}
		0 & 1 & 0\\
		0 & 0 & 1\\
		1 & 0 & 0
	\end{bmatrix}
\]
Note that we will use the term ``main diagonal" to refer to the longest diagonal going from northwest to southeast in a square matrix, the same as when normal matrix coordinates are used.

Similarly to $P_n(\Pi)$, $B_n(\Pi)$ has the permutation matrices in $\Av_n(\Pi)$ as its vertices.
Aside from this, though, there is very little in common between $P_n(\Pi)$ and $B_n(\Pi)$.
One additional similarity is that,  just as in case of $P_n$ and $B_n$ themselves, $P_n(\Pi)$ is the image of $B_n(\Pi)$ via the projection 
\[
	(m_{x,y})_{1 \leq x,y \leq n} \mapsto \sum_{1 \leq x,y \leq n} m_{x,y}e_x =  (y_1,y_2,\dots, y_n)
\]
where $e_i$ denotes the $i^{\rm th}$ standard basis vector in $\RR^n$ and $y_i$ is the unique index such that $m_{x_i,y_i}=1$ for $1\le i\le n$.
However, unlike the $\Pi$-avoiding permutohedron, we will see that all trivial Wilf equivalences of permutations yield unimodular equivalences of the corresponding polytopes.

\begin{prop}\label{prop:dihedral}
	If $\Pi \subseteq \symm$, then $B_n(f(\Pi))$ is unimodularly equivalent to $B_n(\Pi)$ for any $f$ in the dihedral group of the square.
\end{prop}

\begin{proof}
	Because $f$ is a dihedral action on the square, there is an obvious corresponding action on the vertices of $B_n(\Pi)$ to obtain the vertices of $B_n(f(\Pi))$.
	This action is a particular permutation of the elements of each matrix, which is itself a unimodular transformation.
	Applying the action to the full polytope $B_n(\Pi)$ results in a unimodular transformation whose image is $B_n(f(\Pi))$.
\end{proof}

This characterization of unimodularly equivalent polytopes allows us to more efficiently study $B_n(\Pi)$, as did Proposition~\ref{prop: Pn equivalence}. However, as the reader will see in the following sections, the analysis of $B_n(\Pi)$ appears to be much more difficult than it was for $P_n(\Pi)$.  So we will content ourselves with describing a few special cases.

We begin again with the most natural starting point: choosing $\Pi$ to be a single element of $\symm_3$.
By Proposition~\ref{prop:dihedral}, there are only two such classes to consider, which are that of $B_n(123)$ and $B_n(132)$.
Although $123$ and $132$ are Wilf equivalent, they are not trivially Wilf equivalent and their corresponding Birkhoff polytopes are not unimodularly equivalent.
Table~\ref{tab:single element data} provides experimental data for these two polytopes when $n$ is small.

\begin{table}
\begin{tabular}{|c|c|c|c|c|} \hline
$P$ & $\dim P$ & $f$-vector of $P$ & $h^*(P)$ & $\Vol P$ \\ \hline
$B_3(123)$ & $4$ & $(1,5,10,10,5,1)$ & $(1)$ & $1$ \\
$B_4(123)$ & $9$ & $(1,14, 83, 275, 565, 752, 654, 363, 120, 20, 1)$ & $(1,4,6,4,1)$ & $16$ \\
$B_5(123)$ & $16$ & ? & ? & $13890$ \\
$B_3(132)$ & $4$ & $(1,5,10,10,5,1)$ & $(1)$ & $1$ \\
$B_4(132)$ & $9$ & $(1, 14, 85, 290, 610, 822, 714, 390, 125, 20, 1)$ & $(1,4,7,5,1)$ & $17$ \\
$B_5(132)$ & $16$ & ? & ? & $21043$ \\ \hline
\end{tabular}
\caption{Data for $B_n(123)$ and $B_n(132)$ for $n = 3, 4, 5$.}\label{tab:single element data}
\end{table}

We can say more about $B_n(\Pi)$ for certain two-element subsets $\Pi \subseteq \symm_3$.
For the first such, we recall a well-known polytope, introduced in \cite{ChanRobbinsYuen}.

\begin{defn}
	The \emph{Chan-Robbins-Yuen polytope} is the polytope in $\RR^{n\times n}$ defined as
\[
		\CRY_n :=  \conv\{ (m_{x,y}) \mid (m_{x,y}) \text{ a permutation matrix and } m_{x,y} = 0 \text{ for all } x \geq n+3-y\}.
	\]
\end{defn}

One of the most fascinating aspects of $\CRY_n$ is that its volume is known to be a product of consecutive Catalan numbers, but this fact has only been established via analytic techniques \cite{ZeilbergerProof}.
It remains an open problem to find a combinatorial proof.  In what follows, we will say that a permutation matrix $M_\sig$ {\em contains} a pattern $\pi$ if $\sig$ contains $\pi$ and similarly for other definitions from pattern theory.

\begin{prop}
	For all $n$ we have $B_n(123,213) = \CRY_n$. 
\end{prop}

\begin{proof}
	To establish the equality, we will show that the polytopes have the same vertex sets.
In fact, we prove the contrapositive: $M_\sig$ contains a $123$ or $213$ pattern if and only if $m_{x,y}=1$ for some $x\ge n+3-y$.  Assume first that we have $m_{x,y}=1$ where $x\ge n+3-y$.  The number of ones in a row below row $y$ is $y-1$.  And the number of ones in a column to the right of column $x$ is $n-x\le y-3$.  It follows that there must be at least two ones below and to the left of $m_{x,y}$.  Thus these three ones form a copy of $123$ or $213$.

For the converse, let $m_{x,y}=1$ be the one which is furthest to the right in any copy of $123$ or $213$.  It follows that all ones to the right of $m_{x,y}$ must be in lower rows, else $m_{x,y}$ is not rightmost.  Since we know there are at least two elements to the left of $m_{x,y}$ which are smaller, the number of columns to the right of column $x$ is bounded by the number of rows below $y$ minus $2$.  Equivalently $n-x\le y-3$ as we wished to prove.
\end{proof}

We next consider $\Pi = \{123,312\}$. First, we will need a lemma which will be helpful for a number of our results.

\begin{lemma}
\label{lem:unimod}
Suppose $n\ge d+1$ and let $P$ be a polytope in $\RR^n$ with vertices $v_1,\dots,v_{d+1}$ of the form
\begin{equation}
\label{eqn:unimod}
v_j=(\overbrace{0,\dots,0}^{j-1},1,*,\dots,*)^T
\end{equation}
for $1\le j\le d+1$ where the stars represent arbitrary integers.  Then $P$ is unimodularly equivalent to $\Delta_d$.

\end{lemma}
\begin{proof}
Let $A$ be the square matrix whose $j$th column is $v_j$ for $1\le j\le d+1$, and is $e_j$ for $j>d+1$.  By definition of the $v_j$, we have that $A$ has a main diagonal of ones with zeros above it.  So $\det A=1$.  Also, by construction, we have $A e_j= v_j$ for $1\le j\le d+1$ and the lemma follows.
\end{proof}

\begin{prop}\label{prop:firstdimension}
	The polytope $B_n(123,312)$ is unimodularly equivalent to $\Delta_{\binom{n}{2}}$.
	Thus, for any $\Pi \subseteq \symm$ containing $123$ and $312$ we have $h^*(B_n(\Pi)) = 1$. 
\end{prop}

\begin{proof}
	Using the proof of \cite[Proposition 13]{SimionSchmidt} and complementation, we see that the permutations $\sigma\in\Av_n(123,312)$ are exactly the elements of $\symm_n$ in the grid class of the matrix
\[
	A=\begin{bmatrix}
		0 & -1 & 0 \\
		-1 & 0 & 0\\
		0 & 0 & -1
	\end{bmatrix}.
\]
So the elements below the main diagonal of $M_{\sig}$ are precisely those corresponding to the $-1$ in the first column of $A$, and once those elements are determined the rest of $M_\sig$ is fixed.  Furthermore, if we know the coordinates of   the southeast-most $m_{x,y}=1$  in the sequence corresponding to that $-1$ in $A$, then the whole sequence is determined because it must be
\begin{equation}
\label{dec}
m_{x,y}, m_{x-1,y+1},\dots,m_{1,x+y-1}.
\end{equation}
To summarize, there is a unique vertex of $B_n(123,312)$ associated with each coordinate pair $(x,y)$ with 
$x+y\le n$, together with a last vertex corresponding to $\sig=n,n-1,\dots,1$ which has no entry below the main diagonal.

To verify the first statement of the proposition, we will use Lemma~\ref{lem:unimod}. To  bring our matrices to the form in equation~(\ref{eqn:unimod}) we reorganize the coordinates according to the map  $\RR^{n\times n}\rightarrow \RR^{n^2}$ given by
	\[
		\begin{bmatrix}
			z_{\binom{n}{2}+1} & * & * & \dots & * & * & * & * \\
			z_{n-1} & * & * & \dots & * & *  & *& *\\
			z_{2n-3} & z_{n-2} & * & \dots & * & *  & * & * \\
			z_{3n-6} & z_{2n-4} & z_{n-3} & \dots & * & *  & * &  *\\
			\vdots & \vdots & \vdots & \ddots & \vdots & \vdots &\vdots & \vdots\\
			z_{\binom{n}{2}-3} & z_{\binom{n}{2} - 7} & z_{\binom{n}{2} - 12} & \dots &z_3 & *& * & *\\
			z_{\binom{n}{2}-1} & z_{\binom{n}{2} - 4} & z_{\binom{n}{2} - 8} & \dots &z_{n+1} & z_2 & * & *\\
			z_{\binom{n}{2}} & z_{\binom{n}{2} - 2} & z_{\binom{n}{2} - 5} & \dots &z_{2n-2} & z_n & z_1 & *
		\end{bmatrix}
		\mapsto (z_1,z_2,\ldots,z_{\binom{n}{2}+1},*,\dots,*)^T,
	\]
where the coordinates with stars are rearranged in a fixed but arbitrary manner.  It is now an easy matter to verify that
the hypothesis of the lemma is satisfied  if $v_j$ is the image of the $M_\sig$ with sequence~(\ref{dec}) ending at position $z_j$ for $1\le j\le \binom{n}{2}$, and for $j=\binom{n}{2}+1$ we take $M_\sig$ to be the matrix with ones on the main diagonal.

	For the second claim, since $B_n(123,312)$ is a unimodular simplex and $B_n(\Pi^{\prime})$ is a subpolytope if $\{123,312\} \subseteq \Pi^{\prime}$, $B_n(\Pi^{\prime})$ is a face of $B_n(123,312)$.
	Thus $B_n(\Pi^{\prime})$ is a lattice simplex of some dimension $k \leq n$, and is unimodular (with respect to its affine span).
	So, using the equivalence we just established, if $\Pi^{\prime}$ contains $123$ and $312$ then $h^*(B_n(\Pi^{\prime})) = 1$.
\end{proof}

\section{The polytopes $B_n(132,312)$ and $\AltB_n(123)$}
\label{sec:132,312 and 123}

The remainder of this paper will be devoted to studying $B_n(132,312)$ and one other class of polytopes.
For this final class we will require some more definitions and notation.
We say a permutation $\sig = a_1\cdots a_n$ is {\em alternating}, or {\em up-down}, if $a_1 < a_2 > a_3 < \cdots$.
In the literature, ``alternating'' sometimes includes {\em down-up} permutations, where the previous inequalities are all reversed.
It is worth noting that alternating permutations may be expressed in terms of {\em vincular patterns}, which are patterns requiring certain elements to occur consecutively.
To indicate this, the portion of the pattern which must be consecutive is underlined. 
For example, $4261573$ contains five instances of the vincular pattern $2\underline{31}$, namely $261$, $461$, $473$, $573$, and $673$;
the subsequence $453$ is not an instance of the vincular pattern $2\underline{31}$ since $5$ and $3$ do not occur consecutively.
The study of vincular patterns was introduced in~\cite{BabsonSteingrimsson} and has since been extended to bivincular patterns, mesh patterns, and other generalizations.
We refer to~\cite{SteingrimssonSurvey} for more information about each of these avoidance classes, including assorted open problems.

Alternating permutations in $\symm_n$ are exactly the elements $\sig = a_1 \cdots a_n \in \Av_n(\underline{\varepsilon 21},\underline{123},\underline{321})$.
The ``$\varepsilon$'' at beginning of the vincular pattern denotes the ``empty permutation'' which has length $0$ and is to be treated as preceding $a_1$.
So $\sig$ containing the pattern $\underline{\varepsilon 21}$ is equivalent to $a_1>a_2$, and avoiding it forces $a_1<a_2$.
In the interest of compact notation, we will write $\AltAv_n(\Pi)$ for 
$\Av_n(\{\underline{\varepsilon 21}$,$\underline{123}$,$\underline{321}\} \cup \Pi)$ and $\AltB_n(\Pi)$ for the analogous variation of $B_n(\Pi)$.

We now introduce the final class of polytopes that we will study, $\AltB_n(123)$. 
We claim that if $n$ is even, then the number of $123$-avoiding alternating permutations is the same in $\symm_n$ and $\symm_{n-1}$.
To see this, note that in any permutation avoiding $123$ the $1$ can not be followed by two elements forming an increasing subsequence.
So if $n$ is even and $\sig = a_1a_2\cdots a_n$ is alternating and $123$-avoiding, then $a_{n-1}=1$.  Furthermore, since $\sig$ avoids $123$ and $a_{n-3}<a_{n-2}$ we must have $a_{n-2}>a_n$.
It follows that standardizing $\sig'=a_1 a_2\dots a_{n-2} a_n$ gives a bijection between the two sets of permutations in question.
Thus, the projection of $\AltB_n(123)$ to $\AltB_{n-1}(123)$, defined by dropping row $n$ and column $n-1$ of the matrices, preserves the Ehrhart polynomial.

To study the Ehrhart theory of $B_n(132,312)$ and $\AltB_n(123)$, we use the following outline:
\begin{enumerate}
	\item Let $P$ be either $B_n(132,312)$ or $\AltB_n(123)$.
	\item In Proposition~\ref{prop:unimodularity}, we will construct a set of simplices, each contained in $P$, such that each simplex $S$ is unimodular with respect to the lattice $\aff(S) \cap \ZZ^{n \times n}$.
	\item Using toric algebra, we will separately construct a triangulation of $P$ in Theorem~\ref{thm:regularity}.
	\item Finally, we will observe that the simplices from Theorem~\ref{thm:regularity} are exactly those formed in Proposition~\ref{prop:unimodularity}.
		Therefore, the triangulations obtained in step $3$ are unimodular with respect to the lattice $\aff(P) \cap \ZZ^{n \times }$. 
\end{enumerate}

%%%%%%%%%%%%%%%%%%%%%%%%
%%%%%%%%%%%%%%%%%%%%%%%%

\subsection{Sublattices of the Weak Order}\label{subsec:weak order}

In order to prove interesting results about the Ehrhart theory of $B_n(132,312)$ and $\AltB_n(123)$, we will first show how the polytopes may be decomposed by putting a partial order on their vertex sets.
These posets (partially ordered sets) are themselves highly structured and interact in a natural way with the geometry of the polytopes. 
We refer the reader to~\cite[Chapter~3]{StanleyVol1Ed2} for the necessary background regarding posets.

Our posets will be constructed using weak Bruhat order.  We will compose permutations from right to left.
A permutation $\sig=a_1 \dots a_n\in\symm_n$ has {\em inversion set}
\[
	\Inv(\sig) = \{ (i,j) \mid \text{$i<j$ and $a_i > a_j$}\},
\]
and {\em inversion value set}
\[
	\Invv(\sig) = \{(a_j,a_i)  \mid \text{$i<j$ and $a_i > a_j$}\}.
\]
The {\em number of inversions} of $\sig$ is $\inv(\sig) =|\Inv(\sig)|=|\Invv(\sig)|$.

The {\em right} (respectively, {\em left}) {\em weak (Bruhat) order} on $\symm_n$ is defined by the cover relations $\sig_1 \lessdot \sig_2$ if there is a simple transposition $s_i$ such that $\sig_1s_i = \sig_2$ (respectively, $s_i\sig_1 = \sig_2$) and $\inv(\sig_2) = \inv(\sig_1) + 1$. 
For example, if $\sig = 2613754$, then $\sig s_3 = 2631754$, and $s_3\sig = 2614753$ and in both cases the number of inversions increases.
The left and right weak orders are isomorphic by the order-preserving map $\sig \mapsto \sig^{-1}$, but it will be important for the reader to keep in mind the distinction between left and right in what follows.

Let $Q_n(\Pi)$ denote the poset obtained by restricting the right weak order to $\Av_n(\Pi)$. 
Similarly define $\AltQ_n(\Pi)$ for the left weak order on $\AltAv_n(\Pi)$.

If $\Pi$ is chosen arbitrarily, then there is no reason to expect these posets to have especially pleasant structure.
We will see, though, that specific choices of $\Pi$ may result in interesting classes of posets.
Figure~\ref{fig:posets} shows the posets $Q_5(132,312)$ and $\AltQ_{8}(123)$.

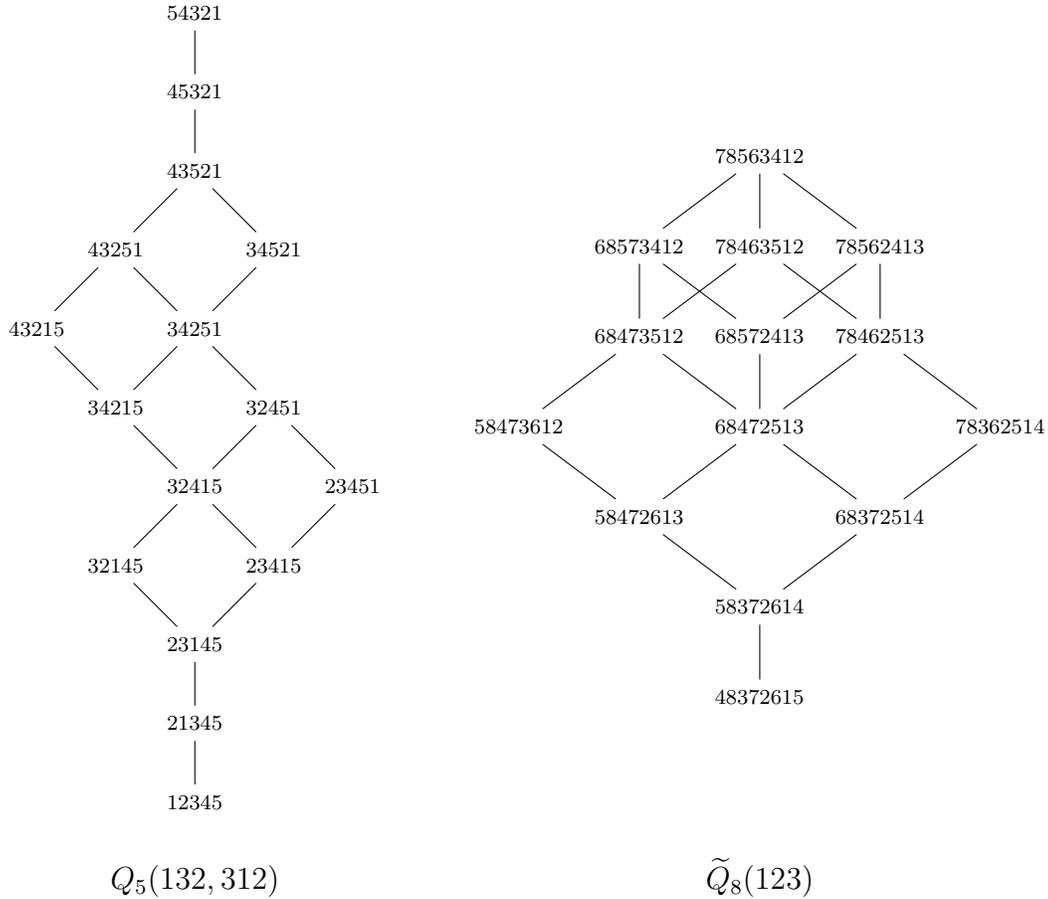
\begin{figure}
	\begin{tikzpicture}[scale=0.7]
	
		\node (A) at (0,0) {\tiny{$12345$}};
		\node (B) at (0,1.5) {\tiny{$21345$}};
		\node (C) at (0,3) {\tiny{$23145$}};
		\node (D) at (-1.5,4.5) {\tiny{$32145$}};
		\node (E) at (1.5,4.5) {\tiny{$23415$}};
		\node (F) at (0,6) {\tiny{$32415$}};
		\node (G) at (3,6) {\tiny{$23451$}};
		\node (H) at (-1.5,7.5) {\tiny{$34215$}};
		\node (I) at (1.5,7.5) {\tiny{$32451$}};
		\node (J) at (-3,9) {\tiny{$43215$}};
		\node (K) at (0,9) {\tiny{$34251$}};
		\node (L) at (-1.5,10.5) {\tiny{$43251$}};
		\node (M) at (1.5,10.5) {\tiny{$34521$}};
		\node (N) at (0,12) {\tiny{$43521$}};
		\node (O) at (0,13.5) {\tiny{$45321$}};
		\node (P) at (0,15) {\tiny{$54321$}};
		
		\draw (A) -- (B) -- (C) -- (D) -- (F) -- (H) -- (J) -- (L) -- (N) -- (O) -- (P);
		\draw (C) -- (E) -- (G) -- (I) -- (K) -- (M) -- (N);
		\draw (E) -- (F) -- (I);
		\draw (H) -- (K) -- (L);
		
		\node at (0,-1.5) {$Q_5(132,312)$};
	\end{tikzpicture} \quad\quad
	\begin{tikzpicture}[scale=0.8]
		\node (A) at (0,-1) {\tiny{$48372615$}};
		\node (B) at (0,0.5) {\tiny{$58372614$}};
		\node (C) at (2,2) {\tiny{$68372514$}};
		\node (D) at (-2,2) {\tiny{$58472613$}};
		\node (E) at (4,3.5) {\tiny{$78362514$}};
		\node (F) at (0,3.5) {\tiny{$68472513$}};
		\node (G) at (-4,3.5) {\tiny{$58473612$}};
		\node (H) at (2,5) {\tiny{$78462513$}};
		\node (I) at (0,5) {\tiny{$68572413$}};
		\node (J) at (-2,5) {\tiny{$68473512$}};
		\node (K) at (2,6.5) {\tiny{$78562413$}};
		\node (L) at (0,6.5) {\tiny{$78463512$}};
		\node (M) at (-2,6.5) {\tiny{$68573412$}};
		\node (N) at (0,8) {\tiny{$78563412$}};
		\node (O) at (0,-4) {$\AltQ_{8}(123)$};
				
		\draw (A) -- (B) -- (C) -- (E) -- (H) -- (K) -- (N) -- (M) -- (J) -- (G) -- (D) -- (B);
		\draw (C) -- (F) -- (J) -- (L) -- (N);
		\draw (D) -- (F) -- (H) -- (L);
		\draw (F) -- (I) -- (K);
		\draw (I) -- (M);
	\end{tikzpicture}
	\caption{Hasse diagrams of posets $Q_5(132,312)$ and  $\AltQ_{8}(123)$.}\label{fig:posets} 
\end{figure}

We will define two well-known posets and prove that these are isomorphic to the posets just defined.
To do so, we first need to introduce certain kinds of Young diagrams.
Given a strictly decreasing partition $\la=(\la_1,\dots,\la_l)$, its 
\emph{shifted Young diagram} is an array of boxes such that row $i$ contains $\la_i$ boxes and begins in column $i$.
Let $M(n)$ denote the poset of shifted Young diagrams with largest part at most $n$, ordered by inclusion that is, $(\la_1,\dots,\la_l) < (\la'_1,\dots,\la'_k)$ if and only if $l \leq k$ and $\la_i \leq \la'_i$ for each $i = 1,\dots,l$.
These are the posets described in Exercise~3.187(a) in~\cite{StanleyVol1Ed2} and studied using linear algebra in~\cite{Proctor}.
In particular, the previously cited exercise establishes that $M(n)$ is a distributive lattice.

For the other class of useful posets, recall that a {\em Dyck path}, $p$, of length $2k$ is a lattice path from $(0,0)$ to $(k,k)$ using steps $(1,0)$ and $(0,1)$, which never goes below the line $y = x$.
We say the steps $(1,0)$ and $(0,1)$ are  {\em east} steps and {\em north} steps, respectively.
Let $D_k$ denote the poset of Dyck paths of length $2k$, where if $d_1,d_2 \in D_k$, then 
$d_1 \le d_2$ if $d_1$ lies weakly to the right of $d_2$.
The posets $D_k$ were shown to be distributive lattices in~\cite{FerrariPinzani}.

For an arbitrary poset $P$, we denote the dual poset by $P^*$.
We may equivalently describe $D_k^*$ as the poset of (left-justified) Young diagrams fitting inside the shape $(k-1,k-2,\dots,1)$, ordering by inclusion.
This equivalence is easily seen by identifying a Dyck path with the region bounded between it, the $y$-axis, and the line $y=k$.

Before proving our isomorphisms, we should make some comments about order polytopes.
Let $Q = \{q_1,\dots,q_s\}$ be a poset, and let $\mathcal{I}_Q$ be the distributive lattice of order ideals of $Q$.
If $I \in \mathcal{I}_Q$, let $\chi_I = (\chi_I(q_1),\dots,\chi_I(q_{s}))$ where
\[
	\chi_I(q_i) = \begin{cases}
				0 & \text{ if } q_i \notin I \\
				1 & \text{ if } q_i \in I
			\end{cases}.
\]
The \emph{order polytope} of $Q$ is
\[
	\OO(Q) = \conv\{ \chi_I \in \RR^{|Q|} \mid I \in \mathcal{I}_Q\}.
\]
Using results from \cite{StanleyTwoPosetPolytopes}, if we could  show that $B_n(132,312)$ or $\AltB_n(123)$ are order polytopes, then certain triangulations, volumes, and other properties of the polytopes would follow immediately.
For example, one might try to show that $B_n(132,312)$  is unimodularly equivalent to $\OO(\Irr(M(n-1)))$, where $\Irr(M(n-1))$ is the poset of irreducibles of $M(n-1)$.
Indeed, this appears to be the case for $n \leq 5$ for $B_n(132,312)$ and $n \leq 8$ for $\AltB_n(123)$ when comparing face vectors.
However, since
$B_n(132,312) \subseteq \RR^{n\times n}$ and $\OO(\Irr(M(n-1))) \subseteq \RR^{2^{n-1}}$, for example, it is not obvious how to find  a specific unimodular equivalence.
One possible approach would be to take some subset $S \subseteq [n] \times [n]$ of size $2^{n-1}$, project $B_n(132,312)$ to $\RR^{2^{n-1}}$ onto coordinates according to the indices in $S$,
and find the reduced form of the matrix $X_S := [v_1 \, v_2 \,  \dots \, v_{2^{n-1}}]$, where $v_1,\dots,v_{2^{n-1}}$ are the projections of the vertices of $B_n(132,312)$. 
One can then check for unimodular equivalence by computing row-reduced echelon forms.
However, an exhaustive search of all possible  $S$ for small $n$ reveals  no choice that works.  It is for this reason that we have resorted to other means.

\begin{question}
Are $B_n(132,312)$ or $\AltB_n(123)$ unimodularly equivalent to order polytopes?
\end{question}

Our next result will provide isomorphisms of both $Q_n(132,312)$ and $\AltQ_n(123)$ with the lattices of certain Young diagrams.

\begin{prop}\label{prop:distributive}
	For all $n$, $Q_n(132,312) \iso M(n-1)$ and $\AltQ_n(123) \iso D_{\lceil n/2 \rceil}^*$.
	Thus, $Q_n(132,312)$ and $\AltQ_n(123)$ are distributive lattices.  Also, the covers in both posets are also covers in weak Bruhat order.
\end{prop}

\begin{proof}  
	First note that the statement about distributive lattices will follow immediately once we have proved the isomorphisms.

	We begin by proving that $Q_n(132,312)\iso M(n-1)$.
	Let $\Des(\sig)$ denote the descent set of $\sig = a_1\dots a_n$, namely,
	\[
		\Des(\sig) = \{i \in [n-1] \mid a_i > a_{i+1}\}. 
	\]
	Note that if $\sig, \tau \in \Av_n(132,312)$ are distinct permutations, then it follows from Lemma~\ref{lem:132,312} that $\Des(\sig) \neq \Des(\tau)$.
	Combined with the fact that $|Q_n(132,312)|=2^{n-1}=|M(n-1)|$, we have that $\Des:Q_n(132,312)\rightarrow M(n-1)$ is a bijection, where we write the descent set in decreasing order and consider it the shape of a shifted Young diagram.
	
To show that $\Des$ and its inverse are order preserving, let $t_{i,j}$ denote the transposition in $\symm_n$ which interchanges $i$ and $j$ where $1\le i < j\le n$.  Given $\sig\in\symm_n$, we consider the set 
\[
	T_L(\sig) =\{ t_{i,j} \mid \inv(t_{i,j}\sig)<\inv(\sig)\}.
\]
We are interested in $T_L(\sig)$ because of the fact~\cite[Proposition 3.1.3]{bjornerbrenti} that $\sig\le\tau$ in right weak order if and only if $T_L(\sig)\subseteq T_L(\tau)$.  It is easy to see that  
\[
	T_L(\sig)=\{ t_{i,j} \mid (i,j)\in\Invv(\sig)\}
\]
and this is the description of $T_L(\sig)$ which we will use.

Now suppose $\Des(\sig)=\lambda$  where $\sig=a_1\dots a_n$ and $\la=(\la_1,\dots,\la_l)$ is the shape of a shifted Young diagram.  We will show that there is a bijection between the $t_{i,j}\in T_L(\sig)$ and the squares of $\la$ where we index those squares using matrix coordinates and also use $\la$ to stand for the set of squares.  An example follows the proof.  Using the description of $\sig$ in Lemma~\ref{lem:132,312} we see that $k\in\Des(\sig)=\la$ if and only if $a_{k+1}$ is on the $-1$ side of the grid (where, by convention, $a_1$ is on the $+1$ side).  And in this case $a_m>a_{k+1}$ for every $m\le k$, whereas elements on the $+1$ side of the grid are the second coordinate in no inversion value pairs.  Also the elements before $a_{k+1}$ form the interval $[a_{k+1}+1,a_{k+1}+k]$.  In addition, the elements on the $-1$ side of the grid are exactly the $l$ smallest elements of $\sig$, where $l$ is the number of parts of $\la$.  So, letting $i=a_{k+1}$,
\[
	\Invv(\sig)=\{(i,j) \mid \text{$1\le i\le l$ and $i+1\le j\le i+k$}\}.
\]
Comparing this to the set of squares of $\la$ which is 
\[
	\{(i,j) \mid \text{$1\le i\le l$ and $i\le j\le i+k-1$}\}
\]
we have the obvious bijection between $T_L(\sig)\leftrightarrow\la$ given by $t_{i,j}\leftrightarrow (i,j-1)$.

We can now show that $\Des$ and $\Des^{-1}$ are order preserving. Suppose $\Des(\sig)=\la$ and $\Des(\tau)=\mu$.  Then $\sig\le\tau$ if and only if $T_L(\sig)\subseteq T_L(\tau)$.  But from the previous paragraph, this is equivalent to $\la\subseteq\tau$ as Young diagrams and that is the partial order on $M(n-1)$.  We also obtain the statement in the theorem about covers.  For, using the previous notation, we have a cover in right weak order if and only if $T_L(\tau)$ is obtained from $T_L(\sigma)$ by adding a single transposition.   But the covers in $M(n-1)$ occur precisely when $\mu$ is obtained from $\la$ by adding a single square.  By the bijection $T_L(\sig)\leftrightarrow\la$, the covers in $M(n-1)$ become covers in $Q(132,312)$.

	Showing $\AltQ_n(123) \iso D_{\lceil n/2 \rceil}^*$ requires a bit more care.
	We will first show that $\AltQ_{2k}(123) \iso \AltQ_{2k-1}(123)$ under the map
	\[
		\varphi(a_1, a_2, \dots, a_{2k}) = a_1-1, \dots, a_{2k-2}-1, a_{2k}-1.
	\]
	That this map is a bijection follows from the discussion when we defined $\AltB_n(123)$.
	Moreover, any $\sig=a_1\dots a_{2k}$ always has $a_{2k-1}=1$. 
	So one will never apply $s_1$ to $\sig$. 
	And applying $s_i$, $i\ge2$, corresponds to acting on $\varphi(\sig)$ with $s_{i-1}$. 
	From this the isomorphism follows.	
	Therefore we may henceforth assume that $n=2k$ for some integer $k$.

Define a function $f: \AltQ_n(123) \to D_k^*$ where the path $f(a_1\dots a_{2k})=p$ is constructed by putting north steps in positions $a_1,a_3,\dots,a_{2k-1}$ and east steps in positions $a_2,a_4,\dots,a_{2k}$.  We must check that $f$ is well defined in that it stays weakly above the line $y=x$.  Note that since the sequences used to define the $N$ and $E$  steps are decreasing, the $i$th step east is in position $a_{2k-2i+2}$.  Since $\pi=a_1\dots a_{2k}$ is also alternating, $a_{2k-2i+2}$ is larger than both $a_{2k-2i+1}$ and all the elements in odd positions to its right.  But these are the positions of the first $i$ north steps, and thus the given $E$ step has sufficiently many $N$ steps preceding it to make the path Dyck.  Directly from its definition, we see that $f$ is injective.  So in must be a bijection since both the domain and range have $C_k$ elements.

We now show that  $f$ is order preserving.  First of all, instead of $T_L(\sig)$ one must use
\[
	T_R(\sig) = \{ t_{i,j} \mid \inv(\sig t_{i,j})<\inv(\sig)\}=\{ t_{i,j} \mid (i,j)\in\Inv(\sig)\}.
\]
One must also be aware that since $\sig$ is alternating with $a_1,a_3,\dots,a_{2k-1}$ and $a_2,a_4,\dots,a_{2k}$ decreasing, then the set of pairs
\[
	\Inv'=\{(2i-1,2j-1) \mid 1\le i< j\le k\} \cup \{(2i,j) \mid 2\le 2i < j \le 2k\}
\]
is contained  in $\Inv(\sig)$ for every $\sig\in\AltQ(123)$.  So we need only consider $\Inv'(\sig):=\Inv(\sig)-\Inv'$. It follows that every pair in $\Inv'(\sig)$ is of the form $(2i-1,2j)$ for some $i< j$.  (We can not have $i=j$ since $\sig$ is alternating.)

Now let $f(\sig)=p$ for a Dyck path $p$ and let $\la$ be the left-justified Young diagram associated with $p$ as described in the paragraph before the proof of this Proposition.  Again, an example follows.  There is a canonical bijection between the squares of $\la$ and pairs consisting of an $N$ step of $p$ together with an earlier $E$ step, where the square $(i,j)$ is the one in the same row as the $N$ step (which must be the $i$th north step reading right-to-left) and the same column of the $E$ step (which must be the $j$th east step reading left-to-right). By the way that the steps of $p$ are labeled, it must be that the $N$ step corresponds to some $a_{2i-1}$ and the $E$ step to some $a_{2(k-j+1)}$.  Furthermore, because $p$ is Dyck, it must be that 
$(2i-1,2(k-j+1))\in\Inv'(\sig)$.  And every element of $\Inv'(\sig)$ is realized this way.  Thus we have a bijection between the squares of 
$\la$ and the elements of $T_R(\sig)$ indexed by elements of $\Inv'(\sig)$where $(i,j)\leftrightarrow t_{2i-1,2(k-j+1)}$.
The rest of the proof is as in the $Q_n(132,312)$ case.
\end{proof}

\begin{figure}

\begin{center}
\raisebox{65pt}{
$
\begin{ytableau}
t_{1,2} &  t_{1,3} & t_{1,4}&  t_{1,5}\\
\none 	  &  t_{2,3} & t_{2,4}\\
\none   & \none    & t_{3,4}
\end{ytableau}
$
}
\qquad
$
\begin{tikzpicture}[scale=.7]
\draw (0,0)--(0,4) (1,2)--(1,4) (2,2)--(2,4) (3,3)--(3,4) (0,2)--(2,2) (0,3)--(3,3) (0,4)--(4,4);
\filldraw (0,1) circle(.05);
\draw(.2,.4) node{$1$};
\draw(.2,1.4) node{$2$};
\draw(.7,1.7) node{$3$};
\draw(1.7,1.7) node{$4$};
\draw(2.2,2.4) node{$5$};
\draw(2.7,2.7) node{$6$};
\draw(3.2,3.4) node{$7$};
\draw(3.7,3.7) node{$8$};
\draw(.5,3.5) node{$t_{1,8}$};
\draw(1.5,3.5) node{$t_{1,6}$};
\draw(2.5,3.5) node{$t_{1,4}$};
\draw(.5,2.5) node{$t_{3,8}$};
\draw(1.5,2.5) node{$t_{3,6}$};
\end{tikzpicture}
$
\end{center}

	\caption{The bijections in Proposition~\ref{prop:distributive}.}\label{fig:diagrams}
\end{figure}
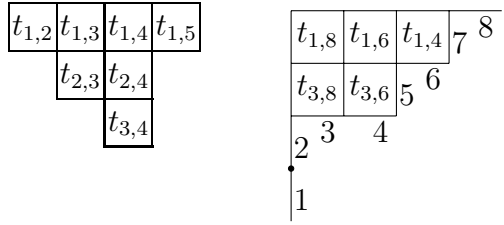

To illustrate the bijection for $Q_n(132,312)$, consider the permutation $\sig=4325167$.  So $\Des(\sig)=(4,2,1)=\la$ whose diagram is displayed on the left  in Figure~\ref{fig:diagrams}.  We also have 
\[
	D_L(\sig)=\{t_{1,2},  t_{1,3}, t_{1,4},  t_{1,5}, t_{2,3}, t_{2,4},  t_{3,4}\}
\]
and each square of $\la$ is labeled with its corresponding transposition.  As for $\AltQ_n(123)$, consider $\sig=78562413$.  So
\[
	\Inv'(\sig)=\{(1,4),\ (1,6),\ (1,8),\ (3,6),\ (3,8)\}.
\]
The path $p$ will have its $N$ steps labeled by $1,2,5,7$ and its $E$ steps labeled by $3,4,6,8$ as on the right in Figure~\ref{fig:diagrams}.  As before, each square of the Young diagram of $p$ is labeled with the corresponding transposition indexed by $\Inv'(\sig)$.

We now return to the general development.
For a general finite distributive lattice $L$ of rank $n$, it is well-known that there exists an $n$-element poset $P$ for which $L \cong J(P)$, where $J(P)$ denotes the lattice of order ideals of $P$.
The poset $P$ can be taken to be the join-irreducible elements of $L$ with order relations inherited from $L$. 
Note that $x\in L$ is join-irreducible if and only if $x$ covers exactly one element.
We denote the poset of join-irreducibles of $L$ by $\Irr(L)$.
To simplify matters, we will identify the join-irreducibles of $Q_n(132,312)$ with the join-irreducibles of $M(n-1)$, and likewise identify the join-irreducibles of $\AltQ_n(123)$ and $D_{\lceil n/2 \rceil}^*$.

Let us now determine the join-irreducibles of our two lattices.
Let $(b,c)$ be the box in row $b$ and column $c$ of a Young diagram $\lambda$.  
(Note that we are taking the diagrams to be in English notation with the largest row on top.)
Call $(b,c)$ an  \emph{inner corner} of  the diagram if  neither $(b+1,c)$ nor $(b,c+1)$ is in $\lambda$.
Using the Young diagram interpretation of our two lattices, an element is join-irreducible precisely when the shape has exactly one inner corner.
Identifying these diagrams with the coordinates of their unique inner corners, the induced partial order on both posets of join-irreducibles is component-wise.
For the remainder of this paper, the join-irreducibles of $Q_n(132,312)$ and $\AltQ_n(123)$ will be identified with the elements of these posets.
See Figure~\ref{fig:extensions} for an example, where for now the label coming after each coordinate pair can be ignored.

%%%%%%%%%%%%%%%%%%%%%%%%
%%%%%%%%%%%%%%%%%%%%%%%%

\subsection{Triangulations, Shellabililty, and EL-labelings}\label{subsec:EL and shellable}

In this section we will use the posets $Q_n(132,312)$ and $\AltQ_n(123)$ to carefully decompose $B_n(132,312)$ and $\AltB_n(123)$.
First, we recall some definitions and concepts in geometry and poset topology.

A {\em polytopal complex} $\FF$ is a finite nonempty collection of polytopes such that 
\begin{enumerate}
	\item if $P \in \FF$, then every face of $P$ is in $\FF$, and
	\item if $P, Q \in \FF$, then $P \cap Q$ is a face of both $P$ and $Q$.
\end{enumerate}
An important polytopal complex is the {\em face complex} $\FF(P)$ of a polytope $P$, whose faces are the faces of $P$.
A polytopal complex $\FF$ is a {\em geometic simplicial complex} if every polytope $P\in\FF$ is a simplex.

A {\em triangulation} of a polytopal complex $\FF$ is a geometric simplicial complex $\Delta$ whose vertices are the vertices of $\FF$ and underlying space equal to the union of the faces of $\FF$,
such that every face of $\Delta$ is contained in a face of $\FF$.
A triangulation of the face complex $\FF(P)$ of a polytope $P$ is simply called a {\em triangulation} of $P$.
Therefore, if $P$ has a unimodular triangulation $\mathcal{T}$, then its normalized volume is equal to the number of maximal simplices in $\mathcal{T}$.

The {\em order complex} $\Delta(Q)$ of a poset $Q$ is the simplicial complex of chains in $Q$.
A simplicial complex is {\em shellable} if its maximal faces are of the same dimension and can be ordered as $F_1,\ldots,F_k$ such that for each $i=1,\ldots, k-1$, 
\[
	F_{i+1} \bigcap \left(\bigcup_{j=1}^i F_j\right)
\]
is a nonempty union of facets of $F_{i+1}$.
A poset is called {\em shellable} if its order complex is shellable.

We will show that $Q_n(132,312)$ and $\AltQ_n(123)$ are shellable by using a particular labeling of the edges in their Hasse diagrams.

If $Q$ is a poset, let $E(Q)$ denote the set
\[
	E(Q) := \{(q_1,q_2) \in Q\times Q \ |\ q_1 \lessdot q_2\},
\] 
thought of as the edges of the Hasse diagram of $Q$.
An {\em edge labeling} of $Q$ by $\ZZ$ is a function $\lambda: E(Q) \to \ZZ$.
A saturated chain $q_0 \lessdot q_1 \lessdot \cdots \lessdot q_k$ in $Q$ is called {\em increasing} if $\lambda(q_0,q_1) < \lambda(q_1,q_2) < \cdots < \lambda(q_{k-1},q_k)$.
An {\em EL-labeling} of a poset $Q$, first introduced in \cite{BjornerShellable}, is an edge labeling such that every interval $[x,y]$ in $Q$ has a unique increasing maximal chain, 
and that chain lexicographically precedes all other maximal chains of $[x,y]$.
Posets admitting an EL-labeling are shellable and are usually referred to as {\em EL-shellable}.

\begin{figure}
	\begin{tikzpicture}

		\node (A) at (0,0) {\tiny{$(1,1), 1$}};
		\node (B) at (0,2) {\tiny{$(2,2), 5$}};
		\node (C) at (0,4) {\tiny{$(3,3), 8$}};
		\node (D) at (0,6) {\tiny{$(4,4), 10$}};
		\node (E) at (1,1) {\tiny{$(1,2), 2$}};
		\node (F) at (1,3) {\tiny{$(2,3), 6$}};
		\node (G) at (1,5) {\tiny{$(3,4), 9$}};
		\node (H) at (2,2) {\tiny{$(1,3), 3$}};
		\node (I) at (2,4) {\tiny{$(2,4), 7$}};
		\node (J) at (3,3) {\tiny{$(1,4), 4$}};

		\draw (A) -- (E) -- (H) -- (J) -- (I) -- (G) -- (D);
		\draw (E) -- (B) -- (F) -- (I);
		\draw (G) -- (C) -- (F) -- (H);
	\end{tikzpicture} \qquad
	\begin{tikzpicture}
		
		\node (E) at (-2,2) {\tiny{$(3,1), 6$}};
		\node (F) at (0,2) {\tiny{$(2,2), 5$}};
		\node (G) at (2,2) {\tiny{$(1,3), 3$}};
		\node (H) at (-1,1) {\tiny{$(2,1), 4$}};
		\node (I) at (1,1) {\tiny{$(1,2), 2$}};
		\node (J) at (0,0) {\tiny{$(1,1), 1$}};
		\node (dummy) at (0,-2) {};

		\draw (E) -- (H) -- (J) -- (I) -- (G);
		\draw (F) -- (I);
		\draw (F) -- (H);
	\end{tikzpicture}
	\caption{The elements of $\Irr(Q_5(132,312))$ and $\Irr(\AltQ_8(123))$ along with their images under  natural labelings.}\label{fig:extensions}
\end{figure}
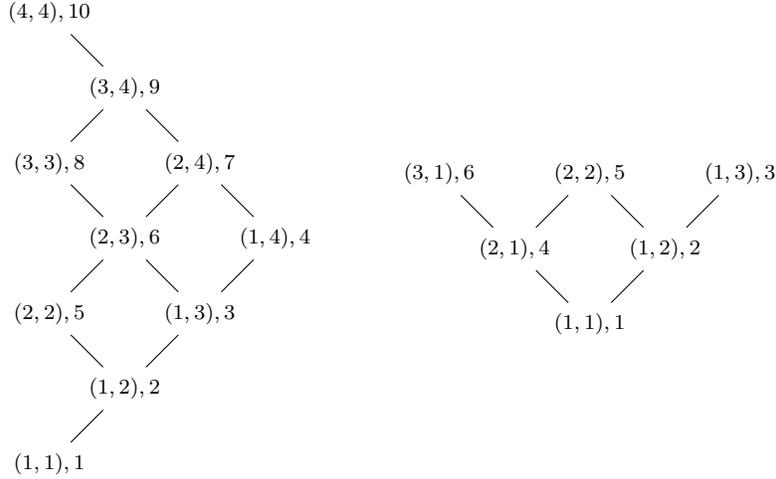

We will use EL-shellable posets to decompose $B_n(132,312)$ and $\AltB_n(123)$ in specific ways in Section~\ref{sec:gorensteinsection}.
Fortunately, specific EL-shellings of $Q_n(132,312)$ and $\AltQ_n(123)$ are available and follow naturally from~\cite{StanleySupersolvable}.
A {\em natural labeling} of a poset $P$ with $|P|=n$ is an order-preserving bijection  $\om:P\rightarrow[n]$.
Let $L$ be a finite distributive lattice so that $L \iso J(P)$ where $P$ is the poset of join-irreducibles, and let $\om$ be a natural labeling of $P$.
Then we have a cover of order ideals $I\lessdot J$ in $L$ if and only if $J  -  I =\{x\}$ for some $x\in P$.  Give the cover the label $\lambda(I,J)=\om(x)$.

\begin{thm}[Stanley, see \cite{StanleySupersolvable}]
	The edge labeling of a finite distributive lattice $L$ constructed above is an EL-labeling for $L$. \qed
\end{thm}

To apply this process we will use the natural labeling of the irreducibles in both of our posets which is obtained by reading the cells $(b,c)$ in each row of the corresponding triangular diagram left to right, starting with the first row and moving down.  Thus in
$\Irr(Q_n(132,321))$ this extension is given by
\[ \om(b,c) = (b-1)n + c + 1 -\binom{b+1}{2} \]
and in $\Irr(\AltQ_n(123))$ for $n$ even by
\[ \om(b,c) = \frac{(b-1)(n-b)}{2}+ c.\]
Alternatively, one can think of both natural labelings as ordering the elements of the poset lexicographically.
Examples of these elements and their associated labels are given in Figure~\ref{fig:extensions}, where the label is displayed beside each element. 
An application of the EL-labeling process appears for $\AltQ_{8}(123)$ in Figure~\ref{fig:altlabeling}.
To simplify notation, we will often identify maximal chains $c: q_0 \lessdot q_1 \lessdot \dots \lessdot q_k$ in $Q_n(132,312)$ and $\AltQ_n(123)$ 
with their sequences of edge labels $\lambda(c)=(\lambda(q_0,q_1), \lambda(q_1,q_2), \dots, \lambda(q_{k-1},q_k))$.

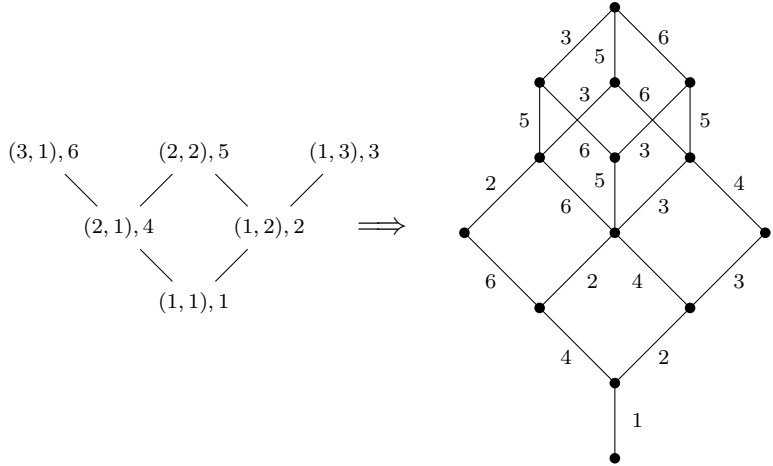
\begin{figure}
	\begin{tikzpicture}

		\node (E) at (-2,2) {\tiny{$(3,1), 6$}};
		\node (F) at (0,2) {\tiny{$(2,2), 5$}};
		\node (G) at (2,2) {\tiny{$(1,3), 3$}};
		\node (H) at (-1,1) {\tiny{$(2,1), 4$}};
		\node (I) at (1,1) {\tiny{$(1,2), 2$}};
		\node (J) at (0,0) {\tiny{$(1,1), 1$}};
		\node (dummy) at (0,-2) {};
		\node (arrow) at (2.5,1) {$\quad\Longrightarrow\quad$};
		   
		\draw (E) -- (H) -- (J) -- (I) -- (G);
		\draw (F) -- (I);
		\draw (F) -- (H);
	\end{tikzpicture}
	\begin{tikzpicture}
	
	\foreach \i in {(0,-1),(0,0),(-1,1),(1,1),(-2,2),(0,2),(2,2),(-1,3),(0,3),(1,3),(-1,4),(0,4),(1,4),(0,5)}
		    {
		    	\fill \i circle [radius=2pt];
		    }
		\draw (0,-1) -- (0,0) -- (-2,2) -- (-1,3) -- (-1, 4) -- (0,5) -- (1,4) -- (1,3) -- (2,2) -- (0,0);
		\draw (-1,1) -- (0,2) -- (1,1);
		\draw (0,2) -- (-1,3) -- (0,4) -- (1,3) -- (0,2);
		\draw (-1,4) -- (0,3) -- (1,4);
		\draw (0,2) -- (0,3);
		\draw (0,4) -- (0,5);
		    
		\node (U) at (0.3,-0.5) {{\tiny $1$}};
		\node (UL) at (-0.65,0.35) {{\tiny $4$}};
		\node (UR) at (0.65,0.35) {{\tiny $2$}};
		\node (ULL) at (-1.65,1.35) {{\tiny $6$}};
		\node (URR) at (1.65,1.35) {{\tiny $3$}};
		\node (ULR) at (-0.3,1.35) {{\tiny $2$}};
		\node (URL) at (0.3,1.35) {{\tiny $4$}};
		\node (ULLR) at (-1.65,2.66) {{\tiny $2$}};
		\node (URRL) at (1.65,2.66) {{\tiny $4$}};
		\node (ULRL) at (-0.65,2.35) {{\tiny $6$}};
		\node (URLR) at (0.65,2.35) {{\tiny $3$}};
		\node (ULRU) at (-0.2,2.7) {{\tiny $5$}};
  		\node (ULRLR) at (-0.4,3.85) {{\tiny $3$}};
		\node (URLRL) at (0.4,3.85) {{\tiny $6$}};
		\node (ULRUL) at (-0.4,3.1) {{\tiny $6$}};
		\node (URLUR) at (0.4,3.1) {{\tiny $3$}};
		\node (URRLU) at (1.2,3.5) {{\tiny $5$}};
		\node (ULLRU) at (-1.2,3.5) {{\tiny $5$}};
		\node (URRLUL) at (0.65,4.6) {{\tiny $6$}};
		\node (ULLRUR) at (-0.65,4.6) {{\tiny $3$}};
		\node (ULRLR) at (-0.2,4.35) {{\tiny $5$}};

	\end{tikzpicture}
	\caption{Producing an edge labeling on $\AltQ_{8}(123)$.}\label{fig:altlabeling}
\end{figure}

We now take a first step in constructing a bridge from purely combinatorial information of these abstract simplicial complexes to geometric information about $B_n(132,312)$ and $\AltB_n(123)$.
One of our main goals is to construct triangulations of $B_n(132,312)$ and $\AltB_n(123)$.
The following proposition only identifies unimodular simplices which potentially form the simplices of triangulations of these polytopes.
The fact that these do form a triangulation will require toric algebra and so will come in Section~\ref{sec:toric algebra}.

\begin{prop}\label{prop:unimodularity}
	Let $f: \Delta(Q_n(132,312)) \to \RR^{n \times n}$ be the function 
	\[
		f(\{\sig_1,\dots,\sig_u\}) = \conv\{M_{\sig_1}, \dots , M_{\sig_u}\},
	\]
	where $M_{\sig_i}$ is the matrix for $\sig_i$.
	The collection
	\[
		\TT_n(132,312) := \{ f(\Gamma)\ |\ \Gamma \in \Delta(Q_n(132,312))\}
	\]
	is a set of simplices contained in $B_n(132,312)$, each $f(\Gamma)$ is unimodular with respect to the affine lattice 
	$\aff(f(\Gamma)) \cap \ZZ^{n\times n}$ and is of dimension $\binom{n}{2}$.
	The collection $\AltT_n(123)$, defined similarly, is a collection of unimodular simplices in $\AltB_n(123)$
of dimension $\binom{n}{\lfloor n/2\rfloor}$.
\end{prop}

\begin{proof}
	First we will focus on $\TT_n(132,312)$.
	Note that it is enough to prove the claim for the simplices in $\TT_n(132,312)$ of maximal dimension, since $\Gamma_1 \subseteq \Gamma_2$ in 
	$\Delta(Q_n(132,312))$ corresponds to an inclusion of faces $f(\Gamma_1) \subseteq f(\Gamma_2)$ in $\TT_n(132,312)$, and faces of unimodular simplices are again unimodular.
	
	Arrange the maximal chains $c_1,\ldots,c_s$ in $Q_n(132,312)$ lexicographically, and let $\Delta^q = \Delta(c_q)$ be the corresponding maximal simplex in $\Delta(Q_n(132,312))$.
	We will prove our claim by induction on $q$.
	
	First consider the $f(\Delta^1)$.
We will use Lemma~\ref{lem:unimod} to show that this is a unimodular simplex.  Note that we have chosen the labeling of the irreducibles  so that $c_1$ starts with the identity permutation $\iota$ and then one proceeds  up the chain by having the element $1$ move from the first position to the last, followed by the element $2$ moving to be the penultimate element, and so forth until one reaches the decreasing permutation.  Thus if one rewrites the coordinates of the $M_\sig$ for $\sig\in c_1$ using the map
\[
		\begin{bmatrix}
			z_{\binom{n}{2}+1} & *  & \dots &  * & * & * \\
			z_{\binom{n}{2}} & *  & \dots & * & *  & *\\
			z_{\binom{n}{2}-2} & z_{\binom{n}{2}-1} & \dots & *  & * & * \\
			\vdots & \vdots  & \ddots & \vdots & \vdots  & \vdots\\
			z_{n} & z_{n+1}  & \dots & z_{2n-3} & * & *\\
			z_1 & z_2& \dots  & z_{n-2} & z_{n-1} & *
		\end{bmatrix}
		\mapsto (z_1,z_2,\ldots,z_{\binom{n}{2}+1},*,\dots,*)^T,
	\]
and let $v_j$ be the image of the element of rank $j-1$ in the chain, then it is easy to check that equation~(\ref{eqn:unimod}) holds.  So we have shown that $f(\Delta^1)$ is unimodular of the correct dimension.
  In particular, if we let $L_1$ be the affine span of $f(\Delta^1)$ then the vectors $f(\beta_{1,r})-f(\iota)$ for $\beta_{1,r} \neq \iota$ in $c_1$ form a basis for the lattice $(L_1-f(\iota)) \cap \ZZ^{\binom{n}{2}}$.
	
	We will perform the induction step by showing that the remaining maximal simplices in $\TT_n(132,312)$ are unimodular transformations of $f(\Delta^1)$. 
	Recall that $Q_n(132,312)$ has an EL-labeling.  So, for $q > 1$, each maximal chain $c_q$ intersects some earlier maximal chain $c_p$ such that they differ by a single element.
	So suppose $c_q$ intersects with $c_p$ such that $\sig \in c_p -  c_q$ and $\sigp \in c_q -  c_p$.
	Then $\sig$ and $\sigp$ are incomparable, and
	\[
		\sig \land \sigp \lessdot \sig,\sigp \lessdot \sig \lor \sigp.
	\]
	So $\sig,\sigp$ can each be obtained from simple transpositions applied to their meet.  And since the interval  
$[\sig \land \sigp, \sig \lor \sigp]$ consists of $4$ elements, these transpositions commute.  If follows that the relationship displayed above is captured by $f$ via
	\begin{equation}
	\label{f(sigma)}
		f(\sig \land \sigp) + f(\sig \lor \sigp) = f(\sig) + f(\sigp).
	\end{equation}
	
	We will use this relationship create a transformation $\varphi: L_p - f(\iota) \to L_q - f(\iota)$ by defining its images on the basis vectors $\beta_{p,r} - f(\iota)$ obtained from the inductive assumption. 
	(The map $\varphi$ implicitly depends on $p$ and $q$ even though that is not reflected in our notation.)
	This function will map $f(\Delta^p) - f(\iota)$ to $f(\Delta^q) - f(\iota)$, and we will show that it is a unimodular transformation.
	It follows that $\Delta^q$ is also unimodular with respect to the affine lattice $L_q \cap \ZZ^{n \times n}$.
	
	For each $r$, set
	\[
		\varphi(f(\beta_{p,r}) - f(\iota)) = f(\beta_{q,r}) - f(\iota).
	\]
	If $\beta_{p,r} \in c_p \cap c_q$, then $\varphi$ acts as the identity on $\beta_{p,r} - f(\iota)$.
	Otherwise, consider the index $t$ such that $\beta_{p,t} = \sig$ and use equation~(\ref{f(sigma)}) to write
	\begin{align*}
		\varphi(f(\beta_{p,t}) - f(\iota)) &= f(\sigp) - f(\iota) \\
			&= [f(\sig \land \sigp) - f(\iota)] + [f(\sig \lor \sigp) -  f(\iota)]- [f(\sig) - f(\iota)] \\
		           &=[f(\beta_{q,t-1}) - f(\iota)] + [f(\beta_{q,t+1}) -  f(\iota)] - [f(\beta_{q,t}) - f(\iota)].
	\end{align*}

	The matrix for $\varphi$  is identical to the identity matrix except in the column corresponding to $\sig$.  
	And in that column, because of the previously displayed equation, the only nonzero entries are a $-1$ on the main diagonal with a $1$ just above it and another $1$  just below.
	So, this matrix is unimodularly equivalent to the identity matrix since it has determinant $-1$, and $\Delta^q$ is a unimodular simplex with respect to $L_q \cap \ZZ^{n \times n}$.
	
	We then apply induction, using the $\varphi$ constructed above.
	Since $\Delta^1$ is unimodular with respect to $L$, so are all of the images of the $\varphi$, and therefore so are all of the $f(\Delta^q)$.
	Thus, $\TT_n(132,312)$ is a collection of unimodular simplices.
		
	The case of $\AltB_n(123)$ is similar.
	First consider $n = 2k$  and $\Delta^1=f(c_1)$.  In order to apply Lemma~\ref{lem:unimod}, read the permutations
$\sig=a_1 a_2\dots a_{2k}\in c_1$ from the bottom of $c_1$ to the top, concentrating only on the subsequence $a_1 a_3\dots a_{2k-1}$.  Recall that we have chosen the labeling $\omega$ to add boxes to the Young diagram row by row, and that $a_{2i+1}$ is the label of the north step at the end of row $i$ of  the Dyck path boundary of the Young diagram.  It follows that $a_1$ will first increase from $k$ to $2k-1$, then $a_3$ will increase from $k-1$ to $2k-3$, and so forth.  This suggests that we use the following map to rewrite the coordinates, where we will just write out the case $n=8$ since the generalization to all even $n$ should then be clear:
\[
		\begin{bmatrix}
			*  &  *  &  *  &  *  &  *  &  *  &  *  &  * \\
			*  &  *  &  *  &  *  &  *  &  *  &  *  &  * \\
			z_3& *  &  *  &  *  &  *  &  *  &  *  &  * \\
			z_2&*  &  *  &  *  &  *  &  *  &  *  &  * \\
			z_1&*  &z_5 &  *  &  *  &  *  &  *  &  * \\
			*  &  *  &z_4  &  *  &  *  &  *  &  *  &  * \\
			*  &  *  &  *  &  *  &z_6  &   * &  *  &  * \\
			*  &  *  &  *  &  *  &  *  &  *  & z_7  & * \\
		\end{bmatrix}
		\mapsto (z_1,z_2,\ldots,z_7,*,\dots,*)^T.
	\]
Now letting $v_i$ be the image of the element at rank $i-1$ in $c_1$ as before completes the proof of unimodularity and the corresponding dimension.

	By creating $\widetilde{\varphi}$ as in the previous case, it follows by induction that $\AltT_{2k}(123)$ is a collection of unimodular simplices in $\AltB_n(123)$.
	As usual, if $n$ is odd, then we use the isomorphism $\AltQ_n(123) \iso \AltQ_{n+1}(123)$ and proceed as in the case of even $n$.
\end{proof}

\subsection{Toric Algebra}\label{sec:toric algebra}

The methods we will use to show $\TT_n(132,312)$ and $\AltT_n(123)$ are unimodular triangulations of their respective polytopes require a bit of algebra background.
Part of the importance of identifying unimodular triangulations is to show when two constructions based on polytopes agree, and we will encounter such a situation in this section; the details of this connection are delayed until the end of the end of the section.
The crucial property of a polytope necessary for the constructions to agree is the following.

\begin{defn}
	A lattice polytope $P \subseteq \RR^n$ is said to have the {\em integer decomposition property} (or to be {\em IDP}) if, for all positive integers $m$ and any $x \in mP \cap \ZZ^n$, there exist $m$ points
	$x_1,\ldots, x_m \in P \cap \ZZ^n$ such that $x = \sum x_i$.
\end{defn}

Much of the exposition that follows is described in \cite[Chapters 4 and 8]{sturmfels}; we reproduce the relevant background below in the interest of self-containment.

First, let $\A = \{l_1,\dots,l_s\} \subseteq \ZZ^n$.
For a field $k$, we may define a subring $k[\A]$ of the ring of Laurent polynomials $k[x_1^{\pm1},\ldots,x_n^{\pm1}]$
by  $k[\A] := k[x^{l_1},\ldots,x^{l_s}]$ where $x^{(v_1,\ldots,v_n)} = \prod x_i^{v_i}$.
Defining $T_{\A} = k[t_1,\dots,t_s]$ and the map $\phi: T_{\A} \to k[\A]$ by $\phi(t_i) = x^{l_i}$, it follows that
\[
	T_{\A} / \ker \phi \iso k[\A].
\]
The ideal $I_{\A} := \ker \phi$ is the {\em toric ideal} of $\A$, and has been studied extensively in part due to its uses in algebraic statistics, algebraic geometry, and convex polytopes.

If $P$ is an integral polytope then we set $\A_P = (P,1) \cap \ZZ^{n+1}$, and 
\[
	k[\cn(P)] := k[x^az^m\ |\ a \in mP \cap \ZZ^n] \subseteq k[x_1^{\pm1},\ldots,x_n^{\pm1},z],
\]
an algebra graded by the exponent of the new variable $z$. So when $P$ is IDP we have $k[\cn(P)] = k[\A_P]$.
However, this equality does not hold if $P$ is not IDP, since then the monoid generated by $\A_P$ does not generate all elements of $\cn(P)\cap \ZZ^{n+1}$.
To remedy this we have to introduce the {\em Hilbert basis} of $\cn(P)$, which is the unique minimal-cardinality set $\HH \subseteq \cn(P) \cap \ZZ^{n+1}$ such that every lattice point of $\cn(P)$ is a $\ZZ_{\geq 0}$-linear combination of elements of $\HH$.
The existence and uniqueness of the Hilbert basis can be proved using the Hilbert Basis Theorem.

This allows us to define the {\em toric ideal} $I_P$ of a polytope $P$:
Suppose the Hilbert basis of $\cn(P)$ is $\HH=\{(v_1,w_1),\dots,(v_r,w_r)\} \subseteq \ZZ^n \times \ZZ$. 
We have
\[
	T_{\HH} / I_P \iso k[\cn(P)],
\]
where $I_P = \ker \phi$ is the toric ideal of $P$.
So, if $P$ is IDP, then $I_P = I_{\A_P}$, but in general we only have $I_P\supseteq I_{\A_P}$.

If there is some $\nu = (\nu_1,\dots,\nu_n) \in \RR^n$ 
such that $\nu^Tl_i = 1$ for each $l_i \in \A$, we call $\A$ a {\em point configuration}, or simply a {\em configuration} if there is no risk of confusion.
When $\A$ is a configuration,  then the {\em positive span}
\[
	\pos(\A) := \left\{ \sum_{i=1}^s \lambda_il_i\ |\ \lambda_i \geq 0 \text{ for all } i \right\} \subseteq \RR^n
\]
is a polyhedral cone (differing from $\cn(\A)\subseteq \RR^{n+1}$) containing no positive-dimensional subspace, so a Hilbert basis exists.
If $\A$ is not a configuration, then no such $\nu$ exists.
In this case, $\pos(\A)$ is still a cone but now contains a nontrivial subspace, so a Hilbert basis does not exist since a minimal generating set of $\pos(\A) \cap \ZZ^n$ is no longer unique.
Note that for any polytope $P$ in $\RR^n$, the set $\A_P$ is a configuration since it satisfies $e_{n+1}^Tv = 1$ for each $v \in \A_P$.

Techniques from toric algebra will provide the tools for a critical step in proving that $B_n(132,312)$ and $\AltB_n(123)$ are IDP by showing that the collections of simplices introduced in the previous section actually form unimodular triangulations of their respective polytopes.  In particular,
when $P$ is one of these polytopes, we will use $I_{\A_P}$ to identify a triangulation of $\conv \A_P$ in which the vertices of the triangulation use only the elements of $\A_P$.
In this case, since $P$ is a subpolytope of $[0,1]^{n \times n}$, it contains no lattice points other than its vertices.
So, $\A_P$ consists exactly of the vertices of $(P,1)$, and a triangulation of $\conv \A_P$ is automatically a triangulation of $(P,1)$, which in turn induces a triangulation of $P$ by projecting each simplex back into $\RR^{n \times n}$.
The triangulation of $P$ will be unimodular with respect to the lattice generated by $\ZZ$-linear combinations of the elements of $P$.
Observing that this triangulation consists exactly of the simplices in $\TT_n(132,312)$ (respectively, $\AltT_n(123)$), Proposition~\ref{prop:unimodularity} will show that the triangulations are unimodular with respect to the affine lattice $B_n(132,312) \cap \ZZ^{n\times n}$ (respectively, $\AltB_n(123) \cap \ZZ^{n \times n}$).

Returning to the general development, 
when $S \subseteq \RR^n$ is a unimodular simplex, it is not difficult to show that $\A_S$ is the Hilbert basis of $\cn(S)$.
When $P$ is a general lattice polytope, we only know a priori that $\A_P$ must be contained in the Hilbert basis of $\cn(P)$.
When a triangulation $\TT$ of $P$ is known, each lattice point $x \in \cn(P)$ lies in $\cn(S)$ for some $S \in \TT$.
If $S$ is unimodular, then $x$ may be written as a sum of just the elements in $(S,1) \cap \ZZ^{n+1} \subseteq \A_P$.
Thus, if $\TT$ is a unimodular triangulation, $x$ can always be expressed as a sum of elements in $\A_P$, so $\A_P$ is exactly the Hilbert basis of $\cn(P)$.
Therefore, in this case, any properties of $(\TT,1)$ as a unimodular triangulation with respect to $\aff \A_P \cap \ZZ^{n+1}$ carry over to $\TT$ as a unimodular triangulation of $P$.

Before continuing with toric ideals, let us first recall some additional definitions.
Let $\Delta$ be an abstract simplicial complex on vertex set $\{v_1,\ldots,v_s\}$ and let $T = k[t_1,\ldots,t_s]$.
The {\em Stanley-Reisner} ideal of $\Delta$ is
\[ I_{\Delta} := ( t_{i_1}\cdots t_{i_j}\ |\ \{i_1,\ldots,i_j\} \notin \Delta),\]
where the parentheses represent the ideal of $T$ generated by these monomials.
We use this ideal to define the {\em Stanley-Reisner ring}, $T/I_{\Delta}$, whose monomials are those with support corresponding to faces of $\Delta$.
The numerator of its Hilbert series is called the $h$-polynomial of $\Delta$.
If $P$ is a polytope and $\Delta$ is a unimodular triangulation of $P$, then the $h$-polynomial of $\Delta$ and the $h^*$-polynomial of $P$ coincide.

Note that the Stanley-Reisner ideal of a simplicial complex accounts for the combinatorial structure of the complex and does not inherently reflect any geometric properties.
To overcome this limitation, we will express the Stanley-Reisner ideal as the result of operations on a different ideal, designed with geometric properties in mind.

%%%%%

Suppose $\prec$ is a monomial order on $T$,
that is, a total well-ordering of the monomials of $T$ which respects multiplication.  Consider any ideal $I$  of $T$.
Each $f \in I$ then has an {\em initial} or {\em leading term} with respect to $\prec$, denoted $\init(f)$, which is the term of $f$ that is greatest with respect to $\prec$.
The {\em initial ideal} of $I$ with respect to $\prec$ is the ideal generated by the initial terms of polynomials in $I$, that is,
\[
	\init(I) := ( \init(f)\ |\ f \in I).
\]
A {\em Gr\"obner basis} of $I$ is a finite generating set $\grob$ for $I$ such that $\init(I) = (\init(g)\ |\ g \in \grob)$.
Since $I$ is assumed to be an ideal of a noetherian ring, a Gr\"obner basis always exists and may be computed from a given finite set of generators for $I$ using the well-known Buchberger algorithm.
Say $\grob$ is {\em reduced} if each element has a leading coefficient of $1$ and for any $g_1, g_2 \in \grob$, $\init(g_1)$ does not divide any term of $g_2$.
Given an ideal $I \subseteq T$ and a fixed monomial ordering on $T$, there are many Gr\"obner bases of $I$ but there is exactly one reduced Gr\"obner basis of $I$.

\begin{thm}[{\cite[Corollary 4.4 and Lemma 4.14]{sturmfels}}]
\label{thm:grob}
	Let $\A = \{l_1,\dots,l_s\} \subseteq \ZZ^n$.
	The reduced Gr\"obner basis $\grob$ of the toric ideal $I_{\A}$ consists of binomials of the form $t^u - t^v$ for $u,v \in \ZZ_{\geq 0}^s$, where the monomials $t^u$ and $t^v$ have no variable in common.
	Moreover, the binomials of $\grob$ are homogeneous if and only $\A$ is a configuration.
\end{thm}

There are many additional nice results connecting Gr\"obner bases with combinatorics, one of which involves types of triangulations that we define now.
Suppose $P \subseteq \RR^n$ is an $n$-dimensional lattice polytope and $P \cap \ZZ^n = \{l_1,\ldots,l_s\}$.
Choose a vector $w = (w_1,\ldots,w_s) \in \RR^s$ such that the polytope
\[ P_w := \conv\{(l_1,w_1),\ldots,(l_s,w_s)\} \subseteq \RR^{n+1}\]
is $(n+1)$-dimensional, i.e., $P_w$ does not lie in an affine hyperplane of $\RR^{n+1}$.
Certain facets of $P_w$ have outward-pointing normal vectors with a negative last coordinate; projecting these facets back to $\RR^n$ provides the facets of a polytopal decomposition of $P$.
If the facets are themselves simplices, then the decomposition is a triangulation, and will be denoted $\Upsilon_w(P)$.
Any triangulation that can be obtained in this way for an appropriate choice of $w$ is called {\em regular}.

For a configuration $\A \subseteq \ZZ^n$ of size $s$, there is a close connection between regular triangulations of $\conv(\A)$ and initial ideals of $I_{\A}$.
First, we note that each monomial ordering $\prec$ on $T_{\A} = k[t_1,\dots,t_{s}]$ can be represented by a sufficiently generic {\em weight vector} $w \in \RR^{s}$ such that, for all $u, v \in \ZZ_{\geq 0}^{s}$, $t^u \prec t^v$ if and only if $w^Tu < w^Tv$.
Next, we define the {\em initial complex} $\Delta_{\prec}(I)$ of an ideal $I \subseteq T_{\A}$ with respect to $\prec$ to be the simplicial complex on $[s]$ such that $F$ is a face of $\Delta_{\prec}(I)$ if and only if there is no monomial in $\init(I)$ whose support is $F$.
Using linear programming, one may show the following.

\begin{thm}[{\cite[Theorem 8.3]{sturmfels}}]\label{thm:sturmfels1}
	Let $\A \subseteq \ZZ^n$ be a configuration.
	If $w$ is the weight vector for a monomial order $\prec$ on $T_{\A}$ then the abstract simplicial complex $\Delta_{\prec}(I_{\A})$  is, in fact,  a geometric simplicial complex which is the regular triangulation $\Upsilon_w(\conv(\A))$.
	That is, the set
	\[
		\Upsilon_w(\conv(\A)) = \{\conv(F) \mid F \in \Delta_{\prec}(I_{\A})\}
	\]
	is a regular triangulation of $\conv(\A)$. \qed
\end{thm}

To state the next result we will need, recall that an ideal $I \subset T_{\A}$ having a minimal generating set of monomials is  \emph{squarefree} if no square divides any of these generating monomials.

\begin{thm}[{\cite[Corollaries 8.4 and 8.9]{sturmfels}}]\label{thm:sturmfels2}
	For any monomial order $\prec$ and corresponding weight vector $w$, the radical $\rad(\init(I_{\A}))$ is the Stanley-Reisner ideal of $\Upsilon_w(\conv(\A))$.
	Moreover, $\init(I_{\A})$ is squarefree if and only if $\Upsilon_w(\conv(\A))$ is unimodular with respect to the affine lattice generated by $\ZZ$-linear combinations of lattice points in $\A$. \qed
\end{thm}

The triangulations $\TT_n(132,312)$ and $\AltT_n(123)$ will turn out to have even more properties than those already discussed.
A triangulation is called {\em flag} if all its minimal nonfaces have two elements.
This may be detected algebraically by proving the existence of an initial ideal generated by squarefree quadratic monomials.
We will demonstrate the flag property by taking the vertices of $P = B_n(132,312)$ (respectively, $P = \AltB_n(123)$) and imposing the {\em graded reverse lexicographic (grevlex)} monomial ordering on $T_{\A_P} / I_{\A_P}$ induced from $Q_n(132,312)$ (respectively, $Q = \AltQ_n(123)$) as follows.
Let $T = k[t_1,\dots,t_s]$ and give the variables the total order $t_1 \succ t_2 \succ \dots \succ t_s$.  Given a monomial $t^a$ we let $|a|$ denote the sum of the exponents.
Grevlex extends the order on the variables to all monomials of $k[t_1,\dots,t_s]$ by insisting that $t^a \gglex t^b$ if $|a| > |b|$ or if both $|a| = |b|$ and the rightmost nonzero entry of $a - b$ is negative.
To apply this to $T_{\A_P}$, we must first place an order on the vertices of $P$; for notational convenience, 
since our variables correspond to permutation matrices, we will frequently use the notation $t_{\sig}$ to denote the variable corresponding the matrix for the permutation $\sig$.
To define grevlex order on monomials in these variables, we must first specify the ordering of the variables themselves.
Write $\sigp \glex \sig$ if  $\sig'$ is lexicographically greater than $\sig$ as words.
In this case we  define $t_{\sigp} \gglex t_{\sig}$.

This allows us to define a {\em reverse lexicographic}, or {\em pulling}, triangulation of a lattice polytope $P$, which is any triangulation whose Stanley-Reisner ideal is $\rad\left({\text{in}_{\lglex}(I_P)}\right)$.
Thus, a triangulation of $P$ is reverse lexicographic if its maximal simplices are the projections of the appropriate facets of $P_w$ where $w$ is a weight vector for $\lglex$.
See \cite{LeeTriangulations}, for example, for a recursive geometric description of how to create reverse lexicographic triangulations.

Before we prove the main theorem of this section, we will need two more lemmas.
Recall that a poset is {\em graded} if all of its maximal chains have the same length.

\begin{lemma}
	Let $M_{\sig}$ denote the matrix corresponding to a permutation $\sig$. 
	For any $\sig, \sigp$ that are both in $Q_n(132,312)$ or in $\AltQ_n(123)$, we have
	\begin{equation}\label{eq:comparability}
		M_{\sig} + M_{\sigp} = M_{\sig \land \sigp} +M_{\sig \lor \sigp}.
	\end{equation}
\end{lemma}

\begin{proof}
Our lattices are distributive and thus graded.  Let $r$ and $r'$ be the lengths of  maximal chains in the intervals $[\sigma\land\sigma',\sigma]$ and $[\sigma\land\sigma',\sigma']$, respectively.  Without loss of generality, we can assume $r\ge r'$.  We induct  on the pairs $(r,r')$ in lexicographic order. The case when $r'=0$ in trivial, and the case $(r,r')=(1,1)$ is covered by equation~(\ref{f(sigma)}).  
So take a permutation $\tau$  in the interval $[\sigma\land\sigma',\sigma]$ which is covered by $\sigma$.  Assume $r\ge2$.
First compare $\tau$ and $\sigma'$.
 By choice of $\tau$, we have $\tau\land\sigma'=\sigma\land\sigma'$.  
And since the lattice is semimodular~\cite[Proposition~3.3.2]{StanleyVol1Ed2}, the length of a maximal chain in $[\tau,\tau\lor\sigma']$ is $r'$.
Comparing $\sigma$ and $\tau\lor\sigma'$ we see that, since we are in a distributive lattice, 
\[
	\sigma\land(\tau\lor\sigma')=(\sigma\land\tau)\lor(\sigma\land\sigma')=\tau.
\]
Also clearly $\sigma\lor(\tau\lor\sigma')=\sigma\lor\sigma'$.
Because of the way we have chosen $r$ and $r'$,  we can apply induction to the pair $\tau,\sigma'$ and to the pair $\sigma,\tau\lor\sigma'$, giving
\[
	M_\tau + M_{\sigma'}=M_{\sigma\land\sigma' }+ M_{\tau\lor\sigma'}  \text{ and } 
	M_\sigma +  M_{\tau\lor\sigma'} = M_\tau + M_{\sigma\lor\sigma'}.
\]
Adding these two equations and canceling finishes the proof.
\end{proof}

\begin{lemma}\label{lem:minima}
	For each permutation $\sig = a_1\dots a_n$ define
	\[
		\mu_i(\sig) = \min \{a_1,\dots, a_i\}.
	\]
	Suppose $\sig\le \tau$  in either left or right (weak) Bruhat order.   It follows that $\mu_i(\sig) \leq \mu_i(\tau)$ for all $i$. 
\end{lemma}

\begin{proof}	
	The proof follows quickly by induction if we can prove it for $\sigma \lessdot \tau$.   In this case, $\tau$ was obtained from $\sigma$ by interchanging two elements $a_r$ and $a_s$ where $r<s$ and $a_r<a_s$.
Consider the sets  $A=\{a_1,\dots,a_i\}$ and $B=\{b_1,\dots,b_i\}$.  If $i<r<s$ or $r<s\le i$ then $A=B$ and so the lemma is trivial.  The only remaining possibility  is $r\le i< s$.  But in that case $B$ is obtained from $A$ by replacing $a_r$ with a larger element $a_s$.  So the minimum can only weakly increase in passing from $A$ to $B$ and the proof is complete.
\end{proof}

We are now ready to prove the main result of this section.

\begin{thm}\label{thm:regularity}
	The sets $\TT_n(132,312)$ and $\AltT_n(123)$ are regular, flag, unimodular reverse lexicographic
triangulations of $B_n(132,312)$ and $\AltB_n(123)$, respectively.
\end{thm}

\begin{proof}
	
	First consider $P = B_n(132,312)$, and let $\A = P \cap \ZZ^{n \times n}$, so that $\A_P = \{ (l,1)\ |\ l \in \A\}$.
	Our strategy will be to construct the reduced Gr\"obner basis $\grob$ of $I_{\A_P}$ with respect to $\prec$ which we are taking the grevlex order.
	By Theorem~\ref{thm:sturmfels1}, the initial complex $\Delta_{\prec}(I_{\A_P})$ is a regular triangulation $\Upsilon_w(\A_P)$ of $\conv(\A_P) = (P,1)$, which induces a regular triangulation $\Upsilon_w(P)$ of $P$.
	We will see that $\grob$ consists of binomials whose initial terms are products of distinct pairs of variables corresponding to incomparable elements of $Q_n(132,312)$.
	Thus, by Theorem~\ref{thm:sturmfels2} and the comment directly afterwards, since $\init(I_{\A_P})$ is the Stanley-Reisner ideal for $\Upsilon_w(P)$, the triangulation is flag and unimodular with respect to the affine lattice
 $\aff(P) \cap \ZZ^{n \times n}$.
	By our description of the minimal non-faces of this triangulation, we will know that the simplices in $\Upsilon_w(P)$ are exactly the elements of $\TT_n(132,312)$.
	Since we saw in Proposition~\ref{prop:unimodularity} that each $\Gamma \in \TT_n(132,312)$ is unimodular with respect to $(\aff \Gamma) \cap \ZZ^{n\times n}$, we have that $\TT_n(132,312)$ is actually a triangulation of $P$ with respect to the lattice $(\aff P) \cap \ZZ^{n \times n}$.
	Because of how we defined $\prec$, the triangulation $\TT_n(132,312)$ is reverse lexicographic as well.
	
	We know by Theorem~\ref{thm:grob} that $\grob$ consists of binomials whose structure we will now examine.
	Consider the set of monomials $t_{\sig}t_{\sigp}$ in $T_{\A_P}$ such that $\sig$ and $\sigp$ are incomparable in $Q_n(132,312)$. 
	Because of equation \eqref{eq:comparability}, we know that $t_{\sig}t_{\sigp} - t_{\sig \land \sigp}t_{\sig \lor \sigp} \in I_{\A_P}$.
	By the way we defined $\prec$, on monomials it is a linear extension of the partial order in  $Q_n(132,312)$.  So the smaller of the two terms is the one containing $t_{\sig\land \sigp}$.
	Thus $t_{\sig}t_{\sigp}$ is the initial term of the binomial.
	Since this monomial is quadratic, it must be the initial term of some binomial in $\grob$.
	It quickly follows from the definition of a reduced Gr\"obner basis that there can be no binomial in $\grob$ of degree $3$ or greater whose initial term contains a pair of variables $t_{\sig}, t_{\sig^{\prime}}$ corresponding to incomparable elements $\sig, \sig^{\prime}$ in $Q_n(132,312)$.
	Otherwise, this initial term would be divisible by $t_{\sig}t_{\sig^{\prime}}$, which is itself an initial term of a binomial in $\grob$.

	Now we will show that there are no binomials 
of degree $2$ or greater in $\grob$ with initial term $t_{\sig_1}^{u_1}\dots t_{\sig_r}^{u_r}$ such that $\sig_1 < \dots < \sig_r$ in $Q_n(132,312)$.
	If we assume there is such a binomial, let $t_{\sigp_1}^{v_1}\dots t_{\sigp_{s}}^{v_{s}}$ be the other term in the binomial.
	By Theorem~\ref{thm:grob}  again, we know the binomial is homogeneous and that the two monomials have no common factors. So in the noninitial term
	 there is some variable, which we may take to be $t_{\sigp_1}$, such that that $ t_{\sigp_1} \lglex t_{\sig_i}$ for all $i$.  
    So, by definition of this monomial order, $\sigp_1 \llex \sig_i  $ for all $i$.
	Letting $\sig_1 = a_1\dots a_n$ and $\sigp_1 = c_1\dots c_n$, denote by $j$ the smallest index for which $c_j < a_j$.
	Now
	\[
		\sum_{i=1}^r u_iM_{\sig_i} = \sum_{i=1}^{s} v_iM_{\sigp_i},
	\]
	and we know from Theorem~\ref{thm:grob} yet again that all the $u_i,v_i$ can be taken to be positive.  Thus the entry  with coordinates $(j,c_j)$ in the matrix for the right-hand sum must be positive. Comparison with the left-hand side shows that
	there is some other $\sig_p = b_1\dots b_n$ for which $b_j = c_j$ and $\sig_1 < \sig_p$ in $Q_n(132,312)$.
	
	We will show $c_j$ is equal to some element in $c_1\dots c_{j-1}=a_1\dots a_{j-1}$ and so $\sigp_1$ is not a permutation, the desired contradiction.
	Using Lemma~\ref{lem:minima} and the definition of $j$ we have
	\[
		\min\{a_1, \dots, a_j\} = \mu_j(\sig_1)\le\mu_j(\sig_p)\le b_j = c_j < a_j.
	\]
	But from Lemma~\ref{lem:132,312},  it is clear that any prefix of $\sig_1$ forms an interval.  So the above inequalities show that $c_j\in \{a_1,\dots,a_{j-1}\}=\{c_1,\dots,c_{j-1}\}$.  So $c_j$ is repeated in $\sigp_1$ forcing it not to be a permutation,  the desired contradiction.
	
	We have shown that the binomials in $\grob$ have initial terms that are products of variables that correspond to pairwise incomparable elements in $Q_n(132,312)$.
	So, the initial ideal of $I_{\A_P}$ is radical and therefore, by Theorem~\ref{thm:sturmfels2}, is the Stanley-Reisner ideal of a regular triangulation of $\conv(\A_P)$ which induces a triangulation of $\conv(\A) = P$.
		
	Since the minimal non-edges of the triangulation are pairs of incomparable elements, any chain $\sig_1 < \dots < \sig_r$ in $Q_n(132,312)$ induces a face $\{M_{\sig_1},\dots,M_{\sig_r}\}$ of the triangulation.
	The set of all such faces is exactly $\TT_n(132,312)$, so $\TT_n(132,312)$ is actually a regular triangulation of $B_n(132,312)$. 
	By Proposition~\ref{prop:unimodularity}, this triangulation is unimodular with respect to $(\aff P) \cap \ZZ^{n \times n}$, and since the minimal non-faces are edges, this triangulation is flag. 
	Because this triangulation was the result of taking an initial ideal with respect to a grevlex order, the triangulation is reverse lexicographic.
	
The same proof will work in the case of $\AltB_n(123)$ except during the demonstration that $\sigp_1$ is not a permutation where we used the grid class structure of $\Av_n(132,312)$.  Instead, we show that there is no such $\sigp_1$ in $\AltQ_n(123)$ as follows.  If $c_j$ occurs among $a_1,\dots,a_{j-1}$ then we are done as before.  Otherwise, $a_j$ must occur to the right of $c_j$ in $\sigp_1$.
Recall that applying a simple transposition $s_i$ to an element of $\AltQ_n(123)$ interchanges $i$ which is in odd position with $i+1$ which is in an even position.  It follows that elements in odd positions increase with the partial order while those in even positions decrease.  Since $a_j>b_j$, we must have $j$  even.  If $a_j$ occurs in an even position to the right of $c_j$ in $\sigp_1$, then we have a contradiction since $c_j<a_j$ are the elements in even positions form a decreasing sequence.  If $a_j$ is in an odd position, then $c_{j-1}>a_j$ since the elements in odd positions are also decreasing.  But then $c_{j-1}>a_j>c_j$ which contradicts the fact that $\sigp_1$ is alternating.  This final contradiction finishes the proof.
\end{proof}

As a first application of this theorem, we will compute the dimensions of our polytopes.  Indeed, since the simplices defined in Proposition~\ref{prop:unimodularity} are those of a regular triangulation, their dimensions must be that of the corresponding polytopes.  So we have shown the following.
\begin{cor}
We have $\dim B_n(132,312)=\binom{n}{2}$ and $\dim\AltB_n(123)=\binom{n}{\lfloor n/2\rfloor}$.\hfill\qed
\end{cor}

We can also compute the normalized volumes of $B_n(132,312)$ and $\AltB_n(123)$.

\begin{cor}
	The normalized volume of $B_n(132,312)$ is
	\[
		\Vol B_n(132,312) = \binom{n}{2}!\frac{\prod_{i=1}^{n-1} (i-1)!}{\prod_{i=1}^{n-1} (2i-1)!}
	\]
	The normalized volume of $\AltB_n(123)$ is 
	\[
		\Vol \AltB_n(123) = \binom{k}{2}!\frac{1}{\prod_{i=1}^{k-1}(2i -1)^{k-i}},
	\]
	where $k = \lceil n/2 \rceil$.
\end{cor}
\begin{proof}
	Since $\TT_n(132,312)$ and $\AltT_n(123)$ are unimodular triangulations of $B_n(132,312)$ and $\AltB_n(123)$, the normalized volumes of the polytopes are the total number of maximal simplices in the respective triangulations.
	These are enumerated by counting the maximal chains in $Q_n(132,312)$ and $\AltQ_n(123)$, which are in bijection with shifted SYT of shape $(n-1,\dots,1)$ and left-justified SYT of shape $(k-1,\dots,1)$.
	Such tableaux are counted by the well-known hook formulas, established in \cite{Thrall} and \cite{FrameHookLength}.
\end{proof}

Because the triangulations in Theorem~\ref{thm:regularity} were obtained using the grevlex order, Corollary~2.5 of \cite{StanleyDecompositions} gives us
\[
	h^*(B_n(132,312)) = h(\TT_n(132,312)) = h(\Delta(Q_n(132,312))),
\]
and likewise for $\AltB_n(123)$.
This fact will come into play in the final section when making statements about the components of $h^*$-vectors for our polytopes.

To close this section, we return to the connection between unimodular triangulations and the integer decomposition property.
In particular, we note that not every choice of $\Pi$ produces a polytope $B_n(\Pi)$ with a unimodular triangulation.
If a lattice polytope does have a unimodular triangulation, then it follows quickly that it is also IDP.
To outline why the implication holds, suppose $v_0,\ldots,v_n$ are the vertices of a unimodular simplex $S \subseteq \RR^n$.
Then $x \in mS \cap \ZZ^n$ if and only if $(x,m) \in \cn(S) \cap \ZZ^{n+1}$, where $\cn(S)$ denotes the cone in $\RR^{n+1}$ whose ray generators are $(v_0,1),\dots,(v_n,1)$.
Since $S$ is a simplex, each lattice point in the cone is contained in a single translate of the monoid generated by $\{(v_0,1),\ldots,(v_n,1)\}$, 
where the translates are uniquely determined by the lattice points in the half-open {\em fundamental parallelepiped}
\[
	\Phi_{S} := \{x \in \RR^{n+1}\ |\ x = \sum_{i=0}^n \lambda_i(v_i,1) \text{ where } 0 \leq \lambda_i < 1\}.
\]

For example, given the $1$-dimensional simplex $[-1,1]$, we see that $\Phi_{[-1,1]}$ contains two lattice points, which are $(0,0)$ and $(0,1)$.
So, every lattice point of $\cn([-1,1])$ is contained in exactly one of the translates $\ZZ_{\geq 0}\{(-1,1),(1,1)\}$ or $(0,1) + \ZZ_{\geq 0}\{(-1,1),(1,1)\}$.

The simplex $S$ is unimodular if and only if $\Phi_{S}$ contains exactly one lattice point, which is necessarily $0$.
Thus the lattice points of $\cn(S)$ are exactly the elements of the single monoid $\ZZ_{\geq 0}\{(v_0,1),\dots,(v_n,1)\}$, which forces $S$ to be IDP.
It follows that a polytope with a unimodular triangulation must also be IDP.

Directly proving that a lattice polytope has the integer decomposition property is usually very difficult.
It is more usually established as a byproduct of proving that the polytope has a unimodular triangulation, or simply a unimodular cover.

\begin{conj}
If $\Pi \subseteq \symm_3$, and $B_n(\Pi)$ is nonempty, then $B_n(\Pi)$ is IDP.
\end{conj}

Computer experiments support this conjecture for all choices of $\Pi$ satisfying the given conditions and all $n \leq 5$.
There do exist choices of $\Pi \subseteq \symm$ for which $B_n(\Pi)$ is not IDP, though.
For example, one can verify that
\[
	\begin{bmatrix}
		0 & 1 & 1 & 2 & 0 \\
		1 & 0 & 1 & 0 & 2 \\
		1 & 1 & 0 & 1 & 1 \\
		2 & 0 & 1 & 0 & 1 \\
		0 & 2 & 1 & 1 & 0 \\
	\end{bmatrix}
\]
is a lattice point of $4B_5(2413,3124)$ but cannot be written as a sum of four lattice points from $B_5(2413,3124)$.
This raises the following very broad question.

\begin{question}
	For which choices of $\Pi$ is $B_n(\Pi)$ IDP?
\end{question}

\section{The Ehrhart Theory of $B_n(132,312)$ and $\AltB_n(123)$}\label{sec:gorensteinsection}

The previous section identified shellable, regular, unimodular triangulations of $B_n(132,312)$ and $\AltB_n(123)$ which arose from order complexes of certain distributive lattices;
in this section, we use the EL-labelings of the lattices to study the $h^*$-vectors of the polytopes.
To do so, we require some more definitions and background.

Suppose $P \subseteq \RR^n$ is a lattice polytope containing the origin in its interior.
We say that $P$ is {\em reflexive} if its polar dual
\[
	P^{\vee} := \{x \in \RR^n\ |\ x^Ty \leq 1 \text{ for all } y \in P\}
\]
is also a lattice polytope. 
Any lattice translate of a reflexive polytope is also called reflexive.
A lattice polytope $P$ is said to be {\em Gorenstein} if $kP$ is reflexive for some $k$, called the {\em index}. 
A theorem, due to Stanley, describes exactly the behavior of $h^*$-vectors for Gorenstein polytopes.

\begin{thm}[{\cite[Theorem 4.4]{StanleyHilbert}}]
\label{GorenPalin}
	A lattice polytope is Gorenstein if and only if its $h^*$-vector is palindromic.
\end{thm}

We can use this result together with the following facts about $h^*$-vectors to determine necessary conditions for $P$ to be Gorenstein.
Let $h^*(P) = (h^*_0,\dots,h^*_d)$ where $P$ is any lattice polytope.
We always have $h^*_0 = 1$.
Additionally, as a consequence of Ehrhart-Macdonald reciprocity, the first scaling of $P$ containing an interior lattice point is $(\dim P - d + 1)P$, and the number of interior lattice points in this scaling is $h^*_d$.
Since a Gorenstein polytope has a palindromic $h^*$-vector, then in order to be Gorenstein, the first scaling of $P$ with an interior lattice point must have exactly one such point.

Note that not every set of permutations $\Pi$ will produce a Gorenstein $B_n(\Pi)$. 
Take, for example, $\Pi = \{123,132\}$ and $n = 5$. 
One may verify that the first nonnegative integer scaling $mB_n(123,132)$ containing an interior lattice point occurs when $m=8$, but this scaling contains four interior lattice points rather than the one needed to be Gorenstein.

The main goal of this section will be to prove the following theorem.

\begin{thm}\label{thm:gorenstein}
	For all $n$, $B_n(132,312)$ and $\AltB_n(123)$ are Gorenstein.
\end{thm}

If the hyperplane description of a lattice polytope is known, then proving whether it is Gorenstein is often a straightforward task.
Such a description of $B_n(132,312)$ and $\AltB_n(123)$ has been elusive, though, so we must approach the proof of Theorem~\ref{thm:gorenstein} by showing that their $h^*$ vectors are palindromic and then appealing to Theorem~\ref{GorenPalin}.

One benefit of going through the work of the previous section is that once a Gorenstein polytope is known to have a regular, unimodular triangulation, 
it follows that the $h^*$-vector of the polytope is unimodal in addition to being palindromic~\cite{BrunsRomer}.
Thus, using Theorem~\ref{thm:gorenstein}, the regular unimodular triangulations $\TT_n(132,312)$ and $\AltT_n(123)$, as well as the EL-labelings of $Q_n(132,312)$ and $\AltQ_n(123)$, we will be able to establish that the $h^*$-vectors of these two polytopes are palindromic and unimodal.

In a shellable triangulation (which may be either abstract or geometric) with shelling order $F_1,\dots,F_s$, the {\em restriction} of face $F_j$ is the set $\cR(F_j)$ of vertices $v\in F_j$ such that the facet $F_j-v$ is contained in $F_1\cup\dots\cup F_{j-1}$.  The {\em shelling number} of $F_j$ is $r(F_j)=|\cR(F_j)|$.  The
following result of Stanley shows that the entries of the $h^*$-vector of the polytope being shelled can be computed using shelling numbers.

\begin{prop}[{\cite[Corollary 2.6]{StanleyDecompositions}}]
\label{StanleyDecomp}
	Suppose that $T_1,\ldots, T_k$ is a shelling order of a unimodular triangulation of a lattice polytope $P$. 
	Then the component $h^*_i$ of $h^*(P)$ is equal to the number of simplices $T_j$ such that $r(T_j) = i$. \qed
\end{prop}

When using EL-shellings, there is an easy way to determine the shelling number of a facet, that is, of a maximal chain $c$, from its labeling.  In particular, if 
\[
	\lambda(c)=(\lambda(q_0,q_1),\lambda(q_1,q_2),\dots, \lambda(q_{k-1},q_k))
\]
then $q_m\in \cR(c)$ if and only if we have a descent $\lambda(q_{m-1},q_m)>\lambda(q_m,q_{m+1})$ in $\lambda(c)$.  This is the content of the following lemma of Bj\"orner.
\begin{lemma}[{\cite[Lemma 2.6]{BjornerShellable}}]\label{lem:shelling numbers}
	Let $c$ be a maximal chain of a poset admitting an EL-labeling $\lambda$.  Then
	\[
		r(c)=\des \lambda(c)
	\] 
where $\des$ is the number of descents.\qed
\end{lemma}

The last link in our chain will come from a result in the theory of $(Q,\om)$-partitions as developed by Stanley. A fuller exposition can be found in Chapter 3 of his book~\cite{StanleyVol1Ed2}.
Let $Q$ be a poset with $|Q| = n$, and let $\omega: Q \to [n]$ be a bijection, called a {\em labeling} of $Q$.
We say $f: Q \to \ZZ_{\geq 1}$ is a {\em (dual) $(Q,\omega)$-partition} if 
\begin{enumerate}
	\item[(i)] $f$ is order preserving, and 
	\item[(ii)] if $s < t$ and $\omega(s) > \omega(t)$, then $f(s) < f(t)$.
\end{enumerate}
In a sense one may think of $\omega$ as indicating where strict inequalities of $f$ occur, rather than weak inequalities.
If $\omega$ itself is order-preserving then, as we have already seen,  it is called a {\em natural} labeling of $Q$.
We call $\omega$ {\em dual natural} if its {\em dual labeling} $\overline{\omega}$, defined by the complementation 
$\overline{\omega}(q) = n + 1 - \omega(q)$, is natural.

We will be concerned with the {\em order polynomial} $\Om_{Q,\om}(m)$ of $(Q,\om)$, which is the number of maps $f:Q\rightarrow[m]$ which satisfy conditions (i) and (ii) above.  It can be shown that $\Om_{Q,\om}(m)$ is a polynomial in $m$ of degree $n=|Q|$.  Equivalently, the generating function for the order polynomial must be in the form
\[
	\sum_{m\ge0}  \Om_{Q,\om}(m) t^m  =\frac{A_{Q,\om}(t)}{(1-t)^{n+1}}
\]
where $A_{Q,\om}(t)$ is a polynomial of degree at most $n$ called the {\em Eulerian polynomial} of $(Q,\om)$.  In fact, one can give an explicit description of $A_{Q,\om}(t)$ as follows.  Define the {\em Jordan-H\"older set}  $\LL(Q,\om$  of 
$(Q,\om)$ to be the set of all permutations of the form $w=\om(q_1)\om(q_2)\dots\om(q_n)$ as $q_1,q_2,\dots,q_n$ runs over all linear extensions of $Q$, that is, total orders on $Q$ such that if $q_i<q_j$ in $Q$ then $i<j$.
\begin{thm}[{\cite[Theorem 3.15.8]{StanleyVol1Ed2}}]
\label{JH}
We have
\[
	\sum_{m\ge0}  \Om_{Q,\om}(m) t^m  = \frac{\sum_{w\in\LL(Q,\om)} t^{1+\des w}}{(1-t)^{n+1}}
\]
where $n=|Q|$.\hfill\qed
\end{thm}

Our next goal is to show that under certain conditions $A_{Q,\om}(t)$ is palindromic.  To do this, we will need a trio of results.  Since 
$\Om_{Q,\om}(m)$ is a polynomial it makes sense to talk about its value at a negative argument.  Also, there are many properties of the order polynomial which are true for all natural labelings $\om$.  In this case, we shorten $\Om_{Q,\om}$ to $\Om_Q$ and similarly for other notation.

\begin{thm}[{\cite[Corollaries 3.15.12 and 3.15.18]{StanleyVol1Ed2}}]
\label{Om}
Let $Q$ be a poset with $|Q|=n$ and longest chain of length $l$.
\begin{enumerate}
	\item[(A)] (Reciprocity theorem for order polynomials)  For all $m\in\ZZ$
	\[
		\Om_{Q,\overline{\om}}(m)=(-1)^n \Om_{Q,\om}(-m).
	\]
	\item[(B)]  If $\om$ is natural then
	\[
		\Om_Q(0)=\Om_Q(-1)=\dots=\Om_Q(-l)=0.
	\]
	\item[(C)]  Suppose $\om$ is natural.  The poset $Q$ is graded if and only if 
	\[
		\Om_Q(m)=(-1)^n \Om_Q(-m-l)
	\]
	for all $m\in\ZZ$.\hfill\qed
\end{enumerate}
\end{thm}

\begin{thm}
\label{Qpalin}
	Let $Q$ be a poset and let $\om$ be a natural labeling of $Q$.  Then the Eulerian polynomial $A_Q(t)$ is palindromic if and only if $Q$ is graded.
\end{thm}
\begin{proof}
We will prove the backwards direction as going forwards is similar.  We will  use $\ol{Q}$ as an abbreviation for $(Q,\ol{\om})$.  We also conserve the notation of the previous result.  Using Theorem~\ref{Om} (A),  Theorem~\ref{JH}, and the definition of $\ol{\om}$ in turn we get
\[
	(-1)^n\sum_{m\ge0} \Om_Q(-m) t^m = \sum_{m\ge0} \Om_{\ol{Q}}(m) t^m
	=\frac{\sum_{w\in \LL(\ol{Q})} t^{1+\des w}}{(1-t)^{n+1}}=\frac{\sum_{w\in \LL(Q)} t^{n-\des w}}{(1-t)^{n+1}}.
\]
Also using, in turn,  parts (C) and (B) of the previous result followed by Theorem~\ref{JH}  gives
\[
	(-1)^n\sum_{m\ge0} \Om_Q(-m) t^m =\sum_{m\ge0}\Om_Q(m-l) t^m = t^l \sum_{m\ge0}\Om_Q(m) t^m
	=\frac{\sum_{w\in \LL(Q)} t^{l+1+\des w}}{(1-t)^{n+1}}.
\]
Comparison of the final numerators in the last two series of displayed equalities implies that $A_Q(t)$ is a palindrome, as desired.
\end{proof}

We now have all our tools in place.
The following result, together with Theorem~\ref{GorenPalin}, proves Theorem~\ref{thm:gorenstein}.
\begin{thm}\label{thm:hstarPalin}
	The vectors $h^*(B_n(132,312))$ and $h^*(\AltB_n(123))$ are palindromic for all $n$.
\end{thm}

\begin{proof}
We will only deal with the case of $P=B_n(132,312)$ as $\AltB_n(123)$ is similar.  Let $Q=\Irr(Q_n(132,312))$.  Let $\om$ be the natural labeling of $Q$ used in the EL-labeling $\lambda$ of $Q_n(132,312)$. 
Since $Q$ is graded, we know from Theorem~\ref{Qpalin} that $A_Q(t)$ is palindromic.  So it suffices to show that the coefficient sequence of $A_Q(t)$ equals $h^*(P)$ (where we ignore the constant term of zero in the former).  
Consider the unimodular triangulation of $P$ given in Theorem~\ref{thm:regularity}.
This permits us to apply
Proposition~\ref{StanleyDecomp} and  Lemma~\ref{lem:shelling numbers} to conclude that $h_i^*(P)$ is the number of maximal chains $c$ of $Q_n(132,312)$ with $\des \lambda(c) = i$.  Comparing this with the expression for  $A_Q(t)$ in Theorem~\ref{JH},
we see that it suffices to prove 
\[
	\LL(Q)=\{\lambda(c)\ |\ \text{$c$ a maximal chain in $Q_n(132,312)$}\}.
\]
But this follows since $Q_n(132,312)=J(Q)$ so that linear extensions $q_0,q_1,q_2,\dots$ of $Q$ are in bijective correspondence with maximal chains $q_0 \lessdot q_0\vee q_1 \lessdot q_0\vee q_1\vee q_2 \lessdot \dots$ of $Q_n(132,312)$, and we are using the same function $\om$ to label both the elements of $Q$ and the covers in the chain.
\end{proof}

\begin{cor}
	The vectors $h^*(B_n(132,312))$ and $h^*(\AltB_n(123))$ are unimodal.
\end{cor}
\begin{proof}
	For each $n$, $B_n(132,312)$ and $\AltB_n(123)$ have regular, unimodular triangulations by Theorem~\ref{thm:regularity} and are Gorenstein by Theorem~\ref{thm:gorenstein}.
	By the main result of~\cite{BrunsRomer}, the $h^*$-vectors for each polytope are $h$-vectors for boundaries of simplicial polytopes, that is, they are unimodal.
\end{proof}

\medskip

{\em Acknowledgement.} We thank Richard Stanley for pointing out that the theory of $(P,\omega)$-partitions could be used to obtain Theorem~\ref{Qpalin}. Helpful comments were also given by several anonymous referees.

\bibliographystyle{plain}
\bibliography{references}

\begin{thebibliography}{10}

\bibitem{AthanasiadisBirkhoff}
Christos~A. Athanasiadis.
\newblock Ehrhart polynomials, simplicial polytopes, magic squares and a
  conjecture of {S}tanley.
\newblock {\em J. Reine Angew. Math.}, 583:163--174, 2005.

\bibitem{BabsonSteingrimsson}
Eric Babson and Einar Steingr{\'{\i}}msson.
\newblock Generalized permutation patterns and a classification of the
  {M}ahonian statistics.
\newblock {\em S\'em. Lothar. Combin.}, 44:Art. B44b, 18 pp. (electronic),
  2000.

\bibitem{latte}
Velleda Baldoni, Nicole Berline, Jes{\'u}s~A. De~Loera, Brandon~E. Dutra,
  Matthias K{\"o}ppe, Stanislav Moreinis, Gregory Pinto, Michele Vergne, and
  Jianqiu Wu.
\newblock A userÕs guide for latte integrale v1.7.2, 2013.
\newblock Software package. LattE is available at
  http://www.math.ucdavis.edu/~latte/.

\bibitem{BJMminkowskisum}
Matthias Beck, Katharina Jochemko, and Emily McCullough.
\newblock {$h^*$}-polynomials of zonotopes.
\newblock {\em Trans. Amer. Math. Soc.}
\newblock To appear.

\bibitem{beckpixton}
Matthias Beck and Dennis Pixton.
\newblock The {E}hrhart polynomial of the {B}irkhoff polytope.
\newblock {\em Discrete \& Computational Geometry}, 30(4):623--637, 2003.

\bibitem{BeckRobinsCCDed2}
Matthias Beck and Sinai Robins.
\newblock {\em Computing the continuous discretely}.
\newblock Undergraduate Texts in Mathematics. Springer, New York, second
  edition, 2015.
\newblock Integer-point enumeration in polyhedra, With illustrations by David
  Austin.

\bibitem{BjornerShellable}
Anders Bj{\"o}rner.
\newblock Shellable and {C}ohen-{M}acaulay partially ordered sets.
\newblock {\em Trans. Amer. Math. Soc.}, 260(1):159--183, 1980.

\bibitem{bjornerbrenti}
Anders Bj\"orner and Francesco Brenti.
\newblock {\em Combinatorics of {C}oxeter groups}, volume 231 of {\em Graduate
  Texts in Mathematics}.
\newblock Springer, New York, 2005.

\bibitem{BraunUnimodality}
Benjamin Braun.
\newblock {\em Unimodality problems in Ehrhart theory}, pages 687--711.
\newblock Springer International Publishing, Cham, 2016.

\bibitem{BrunsRomer}
Winfried Bruns and Tim R{\"o}mer.
\newblock {$h$}-vectors of {G}orenstein polytopes.
\newblock {\em J. Combin. Theory Ser. A}, 114(1):65--76, 2007.

\bibitem{DeLoeraOmarBurggraf}
Katherine Burggraf, Jes{\'u}s De~Loera, and Mohamed Omar.
\newblock On volumes of permutation polytopes.
\newblock In {\em Discrete geometry and optimization}, volume~69 of {\em Fields
  Inst. Commun.}, pages 55--77. Springer, New York, 2013.

\bibitem{ChanRobbinsYuen}
Clara~S. Chan, David~P. Robbins, and David~S. Yuen.
\newblock On the volume of a certain polytope.
\newblock {\em Experiment. Math.}, 9(1):91--99, 2000.

\bibitem{CLS2005}
Sylvie Corteel, Sunyoung Lee, and Carla~D. Savage.
\newblock Enumeration of sequences constrained by the ratio of consecutive
  parts.
\newblock {\em S\'em. Lothar. Combin.}, 54A:Art. B54Aa, 12, 2005/07.

\bibitem{DeLoeraKim}
Jes\'us~A. De~Loera and Edward~D. Kim.
\newblock Combinatorics and geometry of transportation polytopes: an update.
\newblock In {\em Discrete geometry and algebraic combinatorics}, volume 625 of
  {\em Contemp. Math.}, pages 37--76. Amer. Math. Soc., Providence, RI, 2014.

\bibitem{DokosEtAl}
Theodore Dokos, Tim Dwyer, Bryan~P. Johnson, Bruce~E. Sagan, and Kimberly
  Selsor.
\newblock Permutation patterns and statistics.
\newblock {\em Discrete Math.}, 312(18):2760--2775, 2012.

\bibitem{Ehrhart}
Eug{\`e}ne Ehrhart.
\newblock Sur les poly\`edres rationnels homoth\'etiques \`a {$n$}\ dimensions.
\newblock {\em C. R. Acad. Sci. Paris}, 254:616--618, 1962.

\bibitem{FerrariPinzani}
Luca Ferrari and Renzo Pinzani.
\newblock Lattices of lattice paths.
\newblock {\em J. Statist. Plann. Inference}, 135(1):77--92, 2005.

\bibitem{FrameHookLength}
J.~S. Frame, G.~de~B. Robinson, and R.~M. Thrall.
\newblock The hook graphs of the symmetric groups.
\newblock {\em Canadian J. Math.}, 6:316--324, 1954.

\bibitem{GrunbaumConvexPolytopes}
Branko Gr\"unbaum.
\newblock {\em Convex polytopes}, volume 221 of {\em Graduate Texts in
  Mathematics}.
\newblock Springer-Verlag, New York, second edition, 2003.
\newblock Prepared and with a preface by Volker Kaibel, Victor Klee and
  G\"unter M. Ziegler.

\bibitem{HohlwegSingletons}
Christophe Hohlweg, Carsten E. M.~C. Lange, and Hugh Thomas.
\newblock Permutahedra and generalized associahedra.
\newblock {\em Adv. Math.}, 226(1):608--640, 2011.

\bibitem{Kapranov}
Mikhail~M. Kapranov.
\newblock The permutoassociahedron, {M}ac {L}ane's coherence theorem and
  asymptotic zones for the {KZ} equation.
\newblock {\em J. Pure Appl. Algebra}, 85(2):119--142, 1993.

\bibitem{LeeTriangulations}
Carl~W. Lee.
\newblock Regular triangulations of convex polytopes.
\newblock In {\em Applied geometry and discrete mathematics}, volume~4 of {\em
  DIMACS Ser. Discrete Math. Theoret. Comput. Sci.}, pages 443--456. Amer.
  Math. Soc., Providence, RI, 1991.

\bibitem{Loday}
Jean-Louis Loday.
\newblock Realization of the {S}tasheff polytope.
\newblock {\em Arch. Math. (Basel)}, 83(3):267--278, 2004.

\bibitem{Onn}
Shmuel Onn.
\newblock Geometry, complexity, and combinatorics of permutation polytopes.
\newblock {\em J. Combin. Theory Ser. A}, 64(1):31--49, 1993.

\bibitem{PitmanStanley}
Jim Pitman and Richard~P. Stanley.
\newblock A polytope related to empirical distributions, plane trees, parking
  functions, and the associahedron.
\newblock {\em Discrete Comput. Geom.}, 27(4):603--634, 2002.

\bibitem{PostnikovEtAl}
Alex Postnikov, Victor Reiner, and Lauren Williams.
\newblock Faces of generalized permutohedra.
\newblock {\em Doc. Math.}, 13:207--273, 2008.

\bibitem{Postnikov2009}
Alexander Postnikov.
\newblock Permutohedra, associahedra, and beyond.
\newblock {\em Int. Math. Res. Not. IMRN}, (6):1026--1106, 2009.

\bibitem{Proctor}
Robert~A. Proctor.
\newblock Solution of two difficult combinatorial problems with linear algebra.
\newblock {\em The American Mathematical Monthly}, 89(10):pp. 721--734, 1982.

\bibitem{ReinerZiegler}
Victor Reiner and G{\"u}nter~M. Ziegler.
\newblock Coxeter-associahedra.
\newblock {\em Mathematika}, 41(2):364--393, 1994.

\bibitem{Schoute}
Pieter~Hendrick Schoute.
\newblock Analytic treatment of the polytopes regularly derived from the
  regular polytopes.
\newblock {\em Verhandelingen der Koninklijke Akademie van Wetenschappen te
  Amsterdam}, 11(3), 1911.

\bibitem{SimionSchmidt}
Rodica Simion and Frank~W. Schmidt.
\newblock Restricted permutations.
\newblock {\em European J. Combin.}, 6(4):383--406, 1985.

\bibitem{StanleySupersolvable}
Richard~P. Stanley.
\newblock Supersolvable lattices.
\newblock {\em Algebra Universalis}, 2:197--217, 1972.

\bibitem{StanleyHilbert}
Richard~P. Stanley.
\newblock Hilbert functions of graded algebras.
\newblock {\em Advances in Math.}, 28(1):57--83, 1978.

\bibitem{StanleyDecompositions}
Richard~P. Stanley.
\newblock Decompositions of rational convex polytopes.
\newblock {\em Ann. Discrete Math.}, 6:333--342, 1980.
\newblock Combinatorial mathematics, optimal designs and their applications
  (Proc. Sympos. Combin. Math. and Optimal Design, Colorado State Univ., Fort
  Collins, Colo., 1978).

\bibitem{StanleyTwoPosetPolytopes}
Richard~P. Stanley.
\newblock Two poset polytopes.
\newblock {\em Discrete Comput. Geom.}, 1(1):9--23, 1986.

\bibitem{StanleyZonotope}
Richard~P. Stanley.
\newblock A zonotope associated with graphical degree sequences.
\newblock In {\em Applied geometry and discrete mathematics}, volume~4 of {\em
  DIMACS Ser. Discrete Math. Theoret. Comput. Sci.}, pages 555--570. Amer.
  Math. Soc., Providence, RI, 1991.

\bibitem{StanleyVol1Ed2}
Richard~P. Stanley.
\newblock {\em Enumerative combinatorics. {V}olume 1}, volume~49 of {\em
  Cambridge Studies in Advanced Mathematics}.
\newblock Cambridge University Press, Cambridge, second edition, 2012.

\bibitem{SteingrimssonSurvey}
Einar Steingr{\'{\i}}msson.
\newblock Some open problems on permutation patterns.
\newblock In {\em Surveys in combinatorics 2013}, volume 409 of {\em London
  Math. Soc. Lecture Note Ser.}, pages 239--263. Cambridge Univ. Press,
  Cambridge, 2013.

\bibitem{sturmfels}
Bernd Sturmfels.
\newblock {\em Gr\"obner {B}ases and {C}onvex {P}olytopes}, volume~8 of {\em
  University Lecture Series}.
\newblock American Mathematical Society, Providence, RI, 1996.

\bibitem{Thrall}
R.~M. Thrall.
\newblock A combinatorial problem.
\newblock {\em Michigan Math. J.}, 1:81--88, 1952.

\bibitem{ZeilbergerProof}
Doron Zeilberger.
\newblock Proof of a conjecture of {C}han, {R}obbins, and {Y}uen.
\newblock {\em Electron. Trans. Numer. Anal.}, 9:147--148 (electronic), 1999.
\newblock Orthogonal polynomials: numerical and symbolic algorithms
  (Legan{\'e}s, 1998).

\bibitem{ZieglerLectures}
G{\"u}nter~M. Ziegler.
\newblock {\em Lectures on polytopes}, volume 152 of {\em Graduate Texts in
  Mathematics}.
\newblock Springer-Verlag, New York, 1995.

\end{thebibliography}

\end{document}